\documentclass[11pt,reqno,a4paper,oneside,11pt]{amsart}%
\usepackage{amsfonts}
\usepackage{mathrsfs}
\usepackage{amssymb}
\usepackage{amsmath}
\usepackage{amsthm}
\usepackage{graphicx}
\usepackage{color}
\usepackage{extarrows}

\providecommand{\U}[1]{\protect\rule{.1in}{.1in}}
\newtheorem{theorem}{Theorem}[section]
\newtheorem{corollary}[theorem]{Corollary}
\newtheorem{lemma}[theorem]{Lemma}

\theoremstyle{definition}
\theoremstyle{remark}
\numberwithin{equation}{section}

\ifx\pdfoutput\relax\let\pdfoutput=\undefined\fi
\newcount\msipdfoutput
\ifx\pdfoutput\undefined\else
\ifcase\pdfoutput\else
\ifx\paperwidth\undefined\else
\ifdim\paperheight=0pt\relax\else\pdfpageheight\paperheight\fi
\ifdim\paperwidth=0pt\relax\else\pdfpagewidth\paperwidth\fi
\fi\fi\fi
\begin{document}
\pagestyle{myheadings}

\begin{center}
{\huge \textbf{  The general solutions to some systems of Sylvester-type quaternion matrix equations with an application}}\footnote{This research was supported by National Natural
Science Foundation of China grant (11971294 and 12171369)
\par
{}* Corresponding author. \par  Email address: wqw@t.shu.edu.cn, wqw369@yahoo.com (Q.W. Wang).}

\bigskip

{ \textbf{Qing-Wen Wang$^{a,b,*}$, Long-Sheng Liu$^{a,c}$}}\\
{\small
\vspace{0.25cm}
 a Department of Mathematics, Shanghai University, Shanghai 200444, P. R. China\\
 b Collaborative Innovation Center for the Marine Artificial Intelligence, Shanghai 200444, P. R. China\\
c School of Mathematics and Computational Science, Anqing Normal University, Anqing 246011, People's Republic of China \quad\quad}
\end{center}
\vspace{1cm}
\begin{quotation}
\noindent\textbf{Abstract:}
Sylvester-type matrix equations have applications in areas including control theory, neural networks, and image processing. In this paper, we establish the necessary and sufficient conditions for the system of Sylvester-type quaternion matrix equations to be consistent and derive an expression of its general solution (when it is solvable). As an application, we investigate the necessary and sufficient conditions for quaternion matrix equations to be consistent and derive a formula for its general solution involving $\eta$-Hermicity. As a special case, we also present the necessary and sufficient conditions for the system of two-sided Sylvester-type quaternion matrix equations to have a solution and derive a formula for its general solution (when it is solvable). Finally, we present an algorithm and an example to illustrate the main results of this paper.\\

\vspace{3mm}

\end{quotation}
\section{\textbf{Introduction}}
Throughout this paper, The field of real numbers is denoted by $\mathbb{R}$. $\mathbb{H}^{m\times n}$ represents the space of all $m\times n$ matrices over $\mathbb{H}$,
	\begin{align*}
	&\mathbb{H}=\{v_{0}+v_{1}\mathbf{i}+v_{2}\mathbf{j}+v_{3}\mathbf{k} | \mathbf{i}^{2}=\mathbf{j}^{2}=\mathbf{k}^{2}=\mathbf{ijk}=-1,v_{0}, v_{1}, v_{2}, v_{3}\in \mathbb{R}\}.
	\end{align*}
	Here, the rank of $A$ is denoted by $r(A)$, while $I$ and $0$ represent an identity matrix and a zero matrix of appropriate sizes, respectively. Term $A^{\ast}$ represents the conjugate transpose of $A$. The Moore-Penrose (M-P) inverse of $A\in \mathbb{H}^{l\times k}$,  $A^{\dagger}$, is defined as the solution of $AYA=A, \  YAY=Y,\  (AY)^{\ast}=AY$ and $(YA)^{\ast}=YA.$ Moreover, $L_{A}=I-A^{\dagger}A$ and $R_{A}=I-AA^{\dagger}$ represent two projectors along $A$. Recall that a quaternion matrix $A$ is called $\eta$-Hermitian if $A=A^{\eta^{\ast}}$, where $A^{\eta^{\ast}}=-\eta A^{\ast}\eta$ and ${\eta}\in \{\mathbf{i}, \mathbf{j}, \mathbf{k}\}$ \cite{F.Z 2011}. It is well-known that $(L_{A})^{\eta^{\ast}}=R_{A^{\eta^{\ast}}}$, $(R_{A})^{\eta^{\ast}}=L_{A^{\eta^{\ast}}}$.
	
	Since Roth \cite{W.E. 1952} first considered the following one-sided Sylvester-type matrix equation in 1952:
	\begin{equation}\label{eq2}
	\begin{aligned}
	AX+YB=C,
	\end{aligned}
	\end{equation}
	which has applications in control theory and singular system control \cite{E.B. 2005}, neural networks \cite{Y.N. 2002},  there have been extensive studies of equation \eqref{eq2}. For example, Baksalary and Kala \cite{J.K. 1980} investigated the necessary and sufficient conditions for the solvability of equation \eqref{eq2} by using the generalized inverses of the matrices involved. Further, Flanders and Wimmer \cite{F.H. 1977} provided an invariant proof of Roth's theorem, while Baksalary and Kala \cite{J.K. 1980} established the necessary and sufficient conditions for the following Sylvester-type matrix equation to be consistent:
	\begin{equation}\label{eq3}
	\begin{aligned}
	C_{3}X_{3}D_{3}+C_{4}X_{4}D_{4}=E_{1}.
	\end{aligned}
	\end{equation}
	\"{O}zg\"{u}ler  \cite{A.B. 1991} studied the necessary and sufficient conditions for the solvability of equation \eqref{eq3} over a principal ideal domain.  Furthermore, Wang \cite{Q.W. 1991} provided some necessary and sufficient conditions for equation \eqref{eq3} to enable a solution over an arbitrary regular ring with identity and obtained an expression for its general solution.
	
	In 1843, Irish mathematician sir William Rowan Hamilton introduced quaternions. It is well known that the quaternion algebra, $\mathbb{H}$, is an associative noncommutative division algebra over  $\mathbb{R}$, which has applications in computer science, orbital mechanics, signal and color image processing, and control theory (\cite{Brahma 2015}, \cite{Bihan 2004}, \cite{Agudelo 2016},  \cite{S.C. 1999}, \cite{L.Q. 2021} \cite{C.C 2011}).
	
	Based on the wide applications of quaternions, interest in Sylvester-type matrix equations has expanded to $\mathbb{H}$, finding many applications, including signal processing, color-image processing and so on (see, e.g., \cite{Z.G. 2019}, \cite{R.L. 2014},  \cite{S.W. 2021}, \cite{S.F. 2013}). Many researchers have studied the Sylvester-type matrix equations over $\mathbb{H}$ ( \cite{H.W. 2013}-\cite{Liu 2022}, \cite{xinliu03}, \cite{M.S 2022}, \cite{R.N 2022}-\cite{Y.F. 2022} ). For example, He et al. \cite{He 2018} investigated some necessary and sufficient conditions for Sylvester-type quaternion matrix equations and derived an expression for their general solution. In addition,  Solvability conditions and the general solution for a system of constrained two-sided Sylvester-type quaternion matrix equations were established by Wang \cite{T30}.  Moreover, Wang et al. \cite{Wang 2009} presented some necessary and sufficient conditions for the Sylvester-type matrix equations
	\begin{equation}\label{eq4}
	\begin{aligned}
	&A_{1}X=C_{1},\ XB_{1}=C_{2},\\
	&A_{2}Y=C_{3},\ YB_{2}=C_{4},\\
	&A_{3}XB_{3}+A_{4}YB_{4}=C_{c}
	\end{aligned}
	\end{equation}
	to provide a common solution and an expression for a general solution to equations \eqref{eq4} over $\mathbb{H}$. In 2022, Liu, et al. \cite{Liu 2021} derived some necessary and sufficient conditions to solve the following  Sylvester-type quaternion matrix equation by using ranks of coefficient matrices and M-P inverses, respectively:
	\begin{equation}\label{eq5}
	\begin{aligned}
	A_{1}X_{1}+X_{2}B_{1}+A_{2}Y_{1}B_{2}+A_{3}Y_{2}B_{3}+A_{4}Y_{3}B_{4}=B.
	\end{aligned}
	\end{equation}
	They also given an expression for a general solution (when it is solvable). Moreover, He and Wang \cite{Z.H. 2021} studied the solvability conditions for the following Sylvester quaternion matrix equations to be consistent using matrix decomposition:
	\begin{equation}\label{new}
	\begin{aligned}
	&A_1Y_1=A_2,\ Y_1B_1=B_2,\\
	A_{11}Y_{1}B_{11}&+A_{22}Y_{2}B_{22}+A_{33}Y_{3}B_{33}=B.
	\end{aligned}
	\end{equation}
	However, to our knowledge, there is no additional information to extend equations \eqref{new} and investigate the necessary and sufficient conditions for equations \eqref{new} to be consistent in terms of M-P inverses and derive an expression for its general solution using these inverses. Motivated by the worked mentioned above and keeping the interest and wide application of matrix equations, in this paper, we extend equations \eqref{new}, i.e., the following Sylvester-type matrix equations:
	\begin{equation}\label{eq1}
	\begin{aligned}
	A_{1}&U=C_{1},\ VB_{1}=D_{1},\\
	A_{2}&X=C_{2},\ XB_{2}=D_{2},\\
	A_{3}&Y=C_{3},\ YB_{3}=D_{3},\\
	A_{4}&Z=C_{4},\ ZB_{4}=D_{4},\\
	E_1U+VF_1+&E_{2}XF_{2}+E_{3}YF_{3}+E_{4}ZF_{4}=C_c.
	\end{aligned}
	\end{equation}
	This is achieved using rank equalities and M-P inverses of some coefficients quaternion matrices in equations \eqref{eq1} and derive a formula for its general solution (when it is solvable), where $A_{i}$, $B_{i}$, $C_{i}$, $D_{i}$, $E_{i}$, $F_{i}\ (i=\overline{1,4})$ and $C_c$ are given matrices, while $X$, $Y$, $Z$ are unknown. It is obvious that the system of matrix equations \eqref{eq1} is an extension of the other equations \eqref{eq2}, \eqref{eq3}, \eqref{eq4}, \eqref{eq5} and \eqref{new}. As a special case of equations \eqref{eq1}, we present some necessary and sufficient conditions for the following  system of two-sided Sylvester-type matrix equations to provide a solution and derive an expression of its general solution (when it is solvable):
	\begin{equation}\label{eq7-1}
	\begin{aligned}
	A_{1}&X=C_{1},\ XB_{1}=D_{1},\\
	A_{2}&Y=C_{2},\ YB_{2}=D_{2},\\
	A_{3}&Z=C_{3},\ ZB_{3}=D_{3},\\
	E_{1}XF_{1}&+E_{2}YF_{2}+E_{3}ZF_{3}=C.
	\end{aligned}
	\end{equation}

  We known that $\eta$-Hermitian matrices have some applications, such as in linear modeling (e.g., \cite{H.Z 2014}, \cite{H.Z 2017}, \cite{C.C 2011}). Many researchers have studied matrix equations involving $\eta$-Hermicity. For instance, He and Wang \cite{H.W. 2013} established the necessary and sufficient conditions for a solution to the following matrix equation:
	\begin{equation}\label{eqn1}
	\begin{aligned}
	B_1XB_{1}^{\eta^{\ast}}+C_1YC_{1}^{\eta^{\ast}}=D_1,
	\end{aligned}
	\end{equation}
	where $X$ and $Y$ are $\eta$-Hermitian. Zhang and Wang \cite{Y.Z. 2013} presented the solvability conditions and the general solution of the following matrix equations:
	\begin{equation}\label{eqn2}
	\begin{aligned}
	A_1X=C_1,\ YB_1=D_1,\\
	A_2XA_{2}^{\eta^{\ast}}+A_3YA_{3}^{\eta^{\ast}}=D_3,
	\end{aligned}
	\end{equation}
	where $X$ and $Y$ are $\eta$-Hermitian.
	
    Furthermore, as an application of equations \eqref{eq1}, we investigate some necessary and sufficient conditions for the following matrix equations to be consistent and derive an expression for its general solution:
	\begin{equation}\label{eq7-2}
	\begin{aligned}
	&A_1U=C_1,\\
	A_{2}X&=C_{2},\ X=X^{\eta^{\ast}},\\
	A_{3}Y&=C_{3},\ Y=Y^{\eta^{\ast}},\\
	A_{4}Z&=C_{4},\  Z=Y^{\eta^{\ast}},\\
	E_1U+(E_1U)^{\eta^{\ast}}+E_{2}&XE_{2}^{\eta^{\ast}}+E_{3}YE_{3}^{\eta^{\ast}}+E_{4}ZE_{4}^{\eta^{\ast}}=C_c.
	\end{aligned}
	\end{equation}
As a special case of equations \eqref{eq7-2}, we establish the necessary and sufficient conditions for the following system of matrix equations to provide a solution and a formula for its general solution, which is an $\eta$-Hermitian solution:
	{\begin{equation}\label{eq7-3}
		\begin{aligned}
		&A_{1}X=C_{1},\ X=X^{\eta^{\ast}},\\
		&A_{2}Y=C_{2},\ Y=Y^{\eta^{\ast}},\\
		&A_{3}Z=C_{3},\  Z=Y^{\eta^{\ast}},\\
		E_{1}XE_{1}^{\eta^{\ast}}&+E_{2}YE_{2}^{\eta^{\ast}}+E_{3}ZE_{3}^{\eta^{\ast}}=C.
		\end{aligned}
		\end{equation}
Clearly, equation \eqref{eqn1} and equations \eqref{eqn2} are special case of equations \eqref{eq7-2}.
		
		The remainder of this article is built up as follows. In Section 2, we present the preliminaries. In Section 3, we establish some necessary and sufficient conditions for the system of matrix equations \eqref{eq1} to have a solution by using the M-P inverses and rank equalities of the quaternion matrices involved. In addition, we provide a formula for its general solution (when it is solvable). As a special case of equations \eqref{eq1}, we also present the solvability conditions and a formula for the general solution of equations \eqref{eq7-1} (when it is solvable).  In Section 4, as an application of equations \eqref{eq1}, we investigate some solvability conditions and the general  solution to equations \eqref{eq7-2}, where  $X,Y,Z$ are $\eta$-Hermitian. Moreover, as a special case of equations \eqref{eq7-2}, we also investigate some solvability conditions and the general  solution to equations \eqref{eq7-3}, which is an $\eta$-Hermitian solution. In section 5, we present an algorithm and an example to illustrate the main results. Finally, we provide a brief conclusions to close the paper in Section 6.
\section{Preliminaries \label{sec2}}

The following lemma is due to Marsaglia and Styan \cite{M.G 1974}, which can be generalized to $\mathbb{H}$.
		\begin{lemma}\label{lem2} \cite{M.G 1974}\ Let $A \in \mathbb{H}^{m\times n},$ $B \in \mathbb{H}^{m\times k}, $ $C \in \mathbb{H}^{l\times n}$, $D\in \mathbb{H}^{j\times k}$ and $E \in\mathbb{H}^{l\times i}$ be given. Then we have the following rank equality:
			$$\ r\begin{pmatrix}
			A&BL_{D}\\
			R_{E}C&0
			\end{pmatrix}=r\begin{pmatrix}
			A&B&0\\
			C&0&E\\
			0&D&0
			\end{pmatrix}-r(D)-r(E).
			$$
		\end{lemma}
		\begin{lemma}\label{lemma2.2}\cite{J.N} Let $A_1$ and $A_2$ be given matrices $\mathbb{H}$. Then $A_1X=A_2$ is solvable if and only if $A_2=A_1A_{1}^{\dagger}A_2$. In this case, the general solution to this equation can be expressed as
			$$X=A_{1}^{\dagger}A_2+L_{A_{1}}U_{1},$$
			where $U_1$ is an any matrix with conformable size over $\mathbb{H}$.	
		\end{lemma}
		\begin{lemma}\label{lemma2.3}\cite{J.N} Let $A_1$ and $A_2$ be given matrices with adequate shapes over $\mathbb{H}$. Then $XA_1=A_2$ is solvable if and only if $A_2=A_2A_{1}^{\dagger}A_1$. In this case, the general solution to this equation can be expressed as
			$$X=A_2A_{1}^{\dagger}+U_{1}R_{A_{1}},$$
			where $U_1$ is an any matrix with conformable size over $\mathbb{H}$.		
		\end{lemma}
		\begin{lemma}\cite{J.N}\label{lem2.2}
			Let $A_{1}\in \mathbb{H}^{m_{1}\times n_{1}},
			B_{1}\in\mathbb{H}^{r_{1}\times
				s_{1}}, C_{1}\in\mathbb{H}^{m_{1}\times r_{1}}$ and $C_{2}\in\mathbb{H}^{n_{1}\times s_{1}}$ be given matrices. Then the system
			\begin{equation}\label{eq2.1}
			\begin{aligned}
			A_{1}X_{1}=C_{1},\quad X_{1}B_{1}=C_{2}
			\end{aligned}
			\end{equation}
			is consistent if and only if
			$$
			R_{A_{1}}C_{1}=0,\quad C_{2}L_{B_{1}}=0,\quad A_{1}C_{2}=C_{1}B_{1}.
			$$
			Under these conditions, a general solution to equations \eqref{eq2.1} can be expressed as
			$$
			X_{1}=A_{1}^{\dagger}C_{1}+L_{A_{1}}C_{2}B_{1}^{\dagger}+L_{A_{1}}U_{1}R_{B_{1}},
			$$
			where $U_{1}$ is an arbitrary matrix of appropriate shape over $\mathbb{H}$.
		\end{lemma}
		\begin{lemma}\label{lem2.3}\cite{Liu 2021}
			Let $A_{i},$ $B_{i}$ and $B$\ $(i=\overline{1,4})$ be given quaternion matrices with appropriate sizes. Put
			\begin{align*}
			&R_{A_{1}}A_{i+1}=A_{ii},\ B_{i+1}L_{B_{1}}=B_{ii} (i=\overline{1,3}),\ T_1=R_{A_{1}}BL_{B_{1}},\ B_{22}L_{B_{11}}=N_{1},\ R_{A_{11}}A_{22}=M_{1},\\
			&  S_{1}=A_{22}L_{M_{1}}, \ C=R_{M_{1}}R_{A_{11}}, C_{1}=CA_{33},\ C_{2}=R_{A_{11}}A_{33},\ C_{3}=R_{A_{22}}A_{33},\ C_{4}=A_{33},\ D=L_{B_{11}}L_{N_{1}},\\
			& D_{1}=B_{33},\ D_{2}=B_{33}L_{B_{22}}, \ D_{3}=B_{33}L_{B_{11}},\ D_{4}=B_{33}D,\ E_{1}=CT_{1},\ E_{2}=R_{A_{11}}T_{1}L_{B_{22}},\\
			&E_{3}=R_{A_{22}}T_{1}L_{B_{11}},\ E_{4}=T_{1}D,\ C_{11}=(L_{C_{2}},\ L_{C_{4}}),\ D_{11}=\left(
			\begin{array}{c}
			R_{D_{1}} \\
			R_{D_{3}} \\
			\end{array}
			\right),\ C_{22}=L_{C_{1}},\ D_{22}=R_{D_{2}},\\
			&C_{33}=L_{C_{3}},\ D_{33}=R_{D_{4}},\ F_{1}=C_{1}^{\dagger}E_{1}D_{1}^{\dagger}+L_{C_{1}}C_{2}^{\dagger}E_{2}D_{2}^{\dagger},\ E_{11}=R_{C_{11}}C_{22},\ E_{22}=R_{C_{11}}C_{33},\ E_{33}=D_{22}L_{D_{11}},\\
			&E_{44}=D_{33}L_{D_{11}},\ F_{2}=C_{3}^{\dagger}E_{3}D_{3}^{\dagger}+L_{C_{3}}C_{4}^{\dagger}E_{4}D_{4}^{\dagger},\ M=R_{E_{11}}E_{22},\ N=E_{44}L_{E_{33}},\ F=F_{2}-F_{1},\ E=R_{C_{11}}FL_{D_{11}},\\ & S=E_{22}L_{M},\ G_{1}=E_{2}-C_{2}C_{1}^{\dagger}E_{1}D_{1}^{\dagger}D_{2},\ F_{11}=C_{2}L_{C_{1}},\ F_{22}=C_{4}L_{C_{3}},\  G_{2}=E_{4}-C_{4}C_{3}^{\dagger}E_{3}D_{3}^{\dagger}D_{4}.
			\end{align*}
			Then following statements are equivalent:
			
			$\mathrm{(1)}$ Equation \eqref{eq5} is consistent.
			
			$\mathrm{(2)}$
			\begin{align*}
			R_{C_{i}}E_{i}=0,\ E_{i}L_{D_{i}}=0\ (i=\overline{1,4}),\ R_{E_{22}}EL_{E_{33}}=0.
			\end{align*}
			
			$\mathrm{(3)}$	
			\begin{align*}
			&r\left(
			\begin{array}{ccccc}
			B & A_{2} & A_{3} & A_{4} & A_{1} \\
			B_{1} & 0 & 0 & 0 & 0 \\
			\end{array}
			\right)=r(B_{1})+r(A_{2},\ A_{3},\ A_{4},\ A_{1}),\\
\end{align*}
\begin{align*}
			&r\left(
			\begin{array}{cccc}
			B & A_{2} & A_{4} & A_{1} \\
			B_{3} & 0 & 0 & 0 \\
			B_{1} & 0 & 0 & 0 \\
			\end{array}
			\right)=r(A_{2},\ A_{4},\ A_{1})+r\left(
			\begin{array}{c}
			B_{3} \\
			B_{1} \\
			\end{array}
			\right),\\
			&r\left(
			\begin{array}{cccc}
			B & A_{3} & A_{4} & A_{1} \\
			B_{2} & 0 & 0 & 0 \\
			B_{1} & 0 & 0 & 0 \\
			\end{array}
			\right)=r(A_{3},\ A_{4},\ A_{1})+r\left(
			\begin{array}{c}
			B_{2} \\
			B_{1} \\
			\end{array}
			\right),\\
			&r\left(
			\begin{array}{ccc}
			B & A_{4} & A_{1} \\
			B_{2} & 0 & 0 \\
			B_{3} & 0 & 0 \\
			B_{1} & 0 & 0 \\
			\end{array}
			\right)=r\left(
			\begin{array}{c}
			B_{2} \\
			B_{3} \\
			B_{1} \\
			\end{array}
			\right)+r(A_{4},\ A_{1}),\\
			&r\left(
			\begin{array}{cccc}
			B & A_{2} & A_{3} & A_{1} \\
			B_{4} & 0 & 0 & 0 \\
			B_{1} & 0 & 0 & 0 \\
			\end{array}
			\right)=r(A_{2},\ A_{3},\ A_{1})+r\left(
			\begin{array}{c}
			B_{4} \\
			B_{1} \\
			\end{array}
			\right),\\
			&r\left(
			\begin{array}{ccc}
			B & A_{2} & A_{1} \\
			B_{3} & 0 & 0 \\
			B_{4} & 0 & 0 \\
			B_{1} & 0 & 0 \\
			\end{array}
			\right)=r\left(
			\begin{array}{c}
			B_{3} \\
			B_{4} \\
			B_{1} \\
			\end{array}
			\right)+r(A_{2},\ A_{1}),\\
			&r\left(
			\begin{array}{ccc}
			B & A_{3} & A_{1} \\
			B_{2} & 0 & 0 \\
			B_{4} & 0 & 0 \\
			B_{1} & 0 & 0 \\
			\end{array}
			\right)=r\left(
			\begin{array}{c}
			B_{2} \\
			B_{4} \\
			B_{1} \\
			\end{array}
			\right)+r(A_{3},\ A_{1}),\\
			&r\left(
			\begin{array}{cc}
			B & A_{1} \\
			B_{2} & 0 \\
			B_{3} & 0 \\
			B_{4} & 0 \\
			B_{1} & 0 \\
			\end{array}
			\right)=r\left(
			\begin{array}{c}
			B_{2} \\
			B_{3} \\
			B_{4} \\
			B_{1} \\
			\end{array}
			\right)+r(A_{1}),\\
			&r\left(
			\begin{array}{ccccccc}
			B & A_{2} & A_{1} & 0 & 0 & 0 & A_{4} \\
			B_{3} & 0 & 0 & 0 & 0 & 0 & 0 \\
			B_{1} & 0 & 0 & 0 & 0 & 0 & 0 \\
			0 & 0 & 0 & -B & A_{3} & A_{1} & A_{4} \\
			0 & 0 & 0 & B_{2} & 0 & 0 & 0 \\
			0 & 0 & 0 & B_{1} & 0 & 0 & 0 \\
			B_{4} & 0 & 0 & B_{4} & 0 & 0 & 0 \\
			\end{array}
			\right)\\
			&=r\left(
			\begin{array}{cc}
			B_{3} & 0 \\
			B_{1} & 0 \\
			0 & B_{2} \\
			0 & B_{1} \\
			B_{4} & B_{4} \\
			\end{array}
			\right)+r\left(
			\begin{array}{ccccc}
			A_{2} & A_{1} & 0 & 0 & A_{4} \\
			0 & 0 & A_{3} & A_{1} & A_{4} \\
			\end{array}
			\right).
			\end{align*}		
			In this case, the solution of equation \eqref{eq5}  can be  expressed as
			\begin{align*}
			&X_{1}=A_{1}^{\dagger}(B-A_{2}Y_{1}B_{2}-A_{3}Y_{2}B_{3}-A_{4}Y_{3}B_{4})-A_{1}^{\dagger}U_{1}B_{1}+L_{A_{1}}U_{2},\\
			&X_{2}=R_{A_{1}}(B-A_{2}Y_{1}B_{2}-A_{3}Y_{2}B_{3}-A_{4}Y_{3}B_{4})B_{1}^{\dagger}+A_{1}A_{1}^{\dagger}U_{1}+U_{3}R_{B_{1}},\\		&Y_{1}=A_{11}^{\dagger}TB_{11}^{\dagger}-A_{11}^{\dagger}A_{22}M_{1}^{\dagger}TB_{11}^{\dagger}-A_{11}^{\dagger}S_{1}A_{22}^{\dagger}TN_{1}^{\dagger}B_{22}B_{11}^{\dagger}\\
			&-A_{11}^{\dagger}S_{1}U_{4}R_{N_{1}}B_{22}B_{11}^{\dagger}+L_{A_{11}}U_{5}+U_{6}R_{B_{11}},\\
			&Y_{2}=M_{1}^{\dagger}TB_{22}^{\dagger}+S_{1}^{\dagger}S_{1}A_{22}^{\dagger}TN_{1}^{\dagger}+L_{M_{1}}L_{S_{1}}U_{7}+U_{8}R_{B_{22}}+L_{M_{1}}U_{4}R_{N_{1}},\\
			&Y_{3}=F_{1}+L_{C_{2}}V_{1}+V_{2}R_{D_{1}}+L_{C_{1}}V_{3}R_{D_{2}},\\
			& or \\\ &Y_{3}=F_{2}-L_{C_{4}}W_{1}-W_{2}R_{D_{3}}-L_{C_{3}}W_{3}R_{D_{4}},
			\end{align*}
			where $T=T_{1}-A_{33}Y_{3}B_{33}$, $U_{i}(i=\overline{1,8})$ are arbitrary matrices with appropriate sizes over $\mathbb{H}$,
			\begin{align*}
			&V_{1}=(I_{m},\ 0)\left[C_{11}^{\dagger}(F-C_{22}V_{3}D_{22}-C_{33}W_{3}D_{33})\right]-(I_{m},\ 0)\left[C_{11}^{\dagger}U_{11}D_{11}-L_{C_{11}}U_{12}\right],\quad \quad\quad \quad\quad\quad \quad \quad \quad \\
			&W_{1}=(0,\ I_{m})\left[C_{11}^{\dagger}(F-C_{22}V_{3}D_{22}-C_{33}W_{3}D_{33})\right]-(0,\ I_{m})\left[C_{11}^{\dagger}U_{11}D_{11}-L_{C_{11}}U_{12}\right],\\
			&W_{2}=\left[R_{C_{11}}(F-C_{22}V_{3}D_{22}-C_{33}W_{3}D_{33})D_{11}^{\dagger}\right]\left(
			\begin{array}{c}
			0 \\
			I_{n} \\
			\end{array}
			\right)+\left[C_{11}C_{11}^{\dagger}U_{11}+U_{21}R_{D_{11}}\right]\left(
			\begin{array}{c}
			0 \\
			I_{n} \\
			\end{array}
			\right),\\
			&V_{2}=\left[R_{C_{11}}(F-C_{22}V_{3}D_{22}-C_{33}W_{3}D_{33})D_{11}^{\dagger}\right]
			\left(                                                       \begin{array}{c}
			I_{n} \\
			0 \\                                                       \end{array}                                                   \right)+\left[C_{11}C_{11}^{\dagger}U_{11}+U_{21}R_{D_{11}}
			\right]
			\left(                                                       \begin{array}{c}
			I_{n} \\
			0 \\                                                       \end{array}                                                   \right),\\
			&V_{3}=E_{11}^{\dagger}FE_{33}^{\dagger}-E_{11}^{\dagger}E_{22}M^{\dagger}FE_{33}^{\dagger}-E_{11}^{\dagger}SE_{22}^{\dagger}FN^{\dagger}E_{44}E_{33}^{\dagger}-E_{11}^{\dagger}SU_{31}R_{N}E_{44}E_{33}^{\dagger}+L_{E_{11}}U_{32}+U_{33}R_{E_{33}},\\
			&W_{3}=M^{\dagger}FE_{44}^{\dagger}+S^{\dagger}SE_{22}^{\dagger}FN^{\dagger}+L_{M}L_{S}U_{41}+L_{M}U_{31}R_{N}-U_{42}R_{E_{44}},
			\end{align*}
			$U_{11}, U_{12}$, $U_{21}$, $U_{31}$, $U_{32}$, $U_{33}$, $U_{41}$ and $U_{42}$ are arbitrary matrices of appropriate sizes over $\mathbb{H}$. $m$ is the column number of $A_{4}$ and $n$ is the row number of $B_{4}$.
		\end{lemma}

\section{Necessary and sufficient conditions for the existence of a solution to equations \eqref{eq1}}
		The goal of this section is to  establish the solvability conditions and a formula of its general solution to equations \eqref{eq1}.
		
		For the convenience, we define some notations the follows: Let $A_{i}$, $B_{i}$, $C_{i}$, $D_{i}$, $E_{i}$, $F_{i}\ (i=\overline{1,4})$, and $C_c$ be given matrices of appropriate sizes over  $\mathbb{H}$. Put

		\begin{align}
		&\begin{aligned}\label{3.1}
		&E_{i}L_{A_{i}}=A_{ii},\ R_{B_{i}}F_{i}=B_{ii}(i=\overline{1,4}),\ B_{j1}=B_{jj}L_{B_{11}} (j=\overline{2,4}),\ T_{1}=C_c-E_{1}A_{1}^{\dagger}C_{1}\\
		&-D_1B_{1}^{\dagger}F_1-\left(\sum_{i=2}^{4}E_{i}(A_{i}^{\dagger}C_{i}+L_{A_{i}}D_{i}B_{i}^{\dagger})F_{i}\right),\ A_{1j}=R_{A_{11}}A_{jj},\ B_{31}L_{B_{21}}=N_{1},\ R_{A_{12}}A_{13}=M_{1},\\ &S_{1}=A_{13}L_{M_{1}},\ R_{A_{11}}T_1L_{B_{11}}=T_{2},
		\end{aligned}
\end{align}
\begin{align}
		&\begin{aligned}\label{3.2}
		&G=R_{M_{1}}R_{A_{12}},\ G_{1}=GA_{14},\ G_{2}=R_{A_{12}}A_{14},\ G_{3}=R_{A_{13}}A_{14},\ G_{4}=A_{14},\ H=L_{B_{21}}L_{N_{1}},\\ &H_{1}=B_{41},\ H_{2}=B_{41}L_{B_{31}},\ H_{3}=B_{41}L_{B_{21}},\ H_{4}=B_{41}H,\ L_{1}=GT_{2},\ L_{2}=R_{A_{12}}T_{2}L_{B_{31}},\\
		&L_{3}=R_{A_{13}}T_{2}L_{B_{21}},\ L_{4}=T_{2}H,
		\end{aligned}\\
		&\begin{aligned}\label{3.3}
		&C_{11}=(L_{G_{2}},\ L_{G_{4}}),\ D_{11}=\left(
		\begin{array}{c}
		R_{H_{1}} \\
		R_{H_{3}} \\
		\end{array}
		\right),\ C_{22}=L_{G_{1}},\ D_{22}=R_{H_{2}},\ C_{33}=L_{G_{3}},\ D_{33}=R_{H_{4}},\\ &E_{11}=R_{C_{11}}C_{22}\ E_{22}=R_{C_{11}}C_{33},\ E_{33}=D_{22}L_{D_{11}},\ E_{44}=D_{33}L_{D_{11}},\ M=R_{E_{11}}E_{22},\\ & N=E_{44}L_{E_{33}},\ F=F_{44}-F_{33},\ E=R_{C_{11}}FL_{D_{11}},\ S=E_{22}L_{M},\ F_{11}=G_{2}L_{G_{1}},
		\end{aligned}\\
		&\begin{aligned}\label{3.4}
		&\ G_{11}=L_{2}-G_{2}G_{1}^{\dagger}L_{1}H_{1}^{\dagger}H_{2},\ F_{22}=G_{4}L_{G_{3}},\  G_{22}=L_{4}-G_{4}G_{3}^{\dagger}L_{3}H_{3}^{\dagger}H_{4},\\
		&F_{33}=G_{1}^{\dagger}L_{1}H_{1}^{\dagger}+L_{G_{1}}G_{2}^{\dagger}L_{2}H_{2}^{\dagger},\ F_{44}=G_{3}^{\dagger}L_{3}H_{3}^{\dagger}+L_{G_{3}}G_{4}^{\dagger}L_{4}H_{4}^{\dagger}.
		\end{aligned}
		\end{align}

		\begin{theorem}\label{the1} Consider \eqref{eq1} with the notation in \eqref{3.1} to \eqref{3.4}. The following statements are equivalent:
			
			$\mathrm{(1)}$ System \eqref{eq1} has a solution.
			
			$\mathrm{(2)}$
			\begin{equation}\label{eqnew}
			\begin{aligned}
			& A_{i}D_{i}=C_{i}B_{i},\ (i=\overline{2,4})
			\end{aligned}
			\end{equation}
			and
			\begin{equation}\label{eqnewcondition}
			\begin{aligned}
			&R_{A_{j}}C_{j}=0,\ D_{j}L_{B_{j}}=0,\ R_{G_{j}}L_{j}=0,\ L_{j}L_{H_{j}}=0 (j=\overline{1,4}),\ R_{E_{22}}EL_{E_{33}}=0.
			\end{aligned}
			\end{equation}
			$\mathrm{(3)}$\ \eqref{eqnew} holds and
			\begin{align}
			&\begin{aligned}\label{3.6}
			&r(C_{i},\ A_{i})=r(A_{i}),\ r\left(
			\begin{array}{c}
			D_{i} \\
			B_{i} \\
			\end{array}
			\right)=r(B_{i})\ (i=\overline{1,4}),
			\end{aligned}\\
			&\begin{aligned}\label{3.7}
			&r\left(
			\begin{array}{cccccc}
			C_c & E_{1} & E_{2} & E_{3} & E_4 & D_1\\
			F_1 &0 &0 &0 &0 & B_1\\
			C_1& A_1& 0& 0&0 &0\\
			C_{2}F_{2} & 0 & A_{2} & 0 & 0& 0\\
			C_{3}F_{3} & 0 & 0 & A_3 & 0& 0\\
			C_{4}F_{4} & 0 & 0 & 0 & A_{4} & 0\\
			\end{array}
			\right)=r\left(
			\begin{array}{cccc}
			E_{1} & E_{2} & E_{3} & E_4\\
			A_{1} & 0 & 0 &0\\
			0 & A_{2} & 0 &0\\
			0 & 0 & A_{3} &0\\
			0 & 0 & 0 &A_{4}\\
			\end{array}
			\right)+r(F_1,\ B_1),
			\end{aligned}\\
			&\begin{aligned}\label{3.8}
			&r\left(
			\begin{array}{cccccc}
			C_c & E_{1} & E_{2} & E_4& E_{3}D_{3} & D_1 \\
			C_1 & A_1 & 0&0&0&0\\
			C_{2}F_{2} & 0 & A_{2} & 0 &0&0\\
			C_{4}F_{4} & 0 & 0 & A_{4} &0&0\\
			F_{3} & 0 & 0 & 0& B_{3}& 0 \\
			F_{1} & 0 & 0 & 0& 0& B_{1}\\
			\end{array}
			\right)=r\left(
			\begin{array}{ccc}
			E_{1} & E_{2} & E_4\\
			A_{1} & 0 &0\\
			0 & A_{2} &0\\
			0 & 0 & A_{4}
			\end{array}
			\right)+r\begin{pmatrix}
			F_3& B_3 &0\\
			F_1&0&B_1
			\end{pmatrix},
			\end{aligned}
\end{align}
\begin{align}
			&\begin{aligned}\label{3.9}
			&r\left(
			\begin{array}{cccccc}
			C_c & E_{1} & E_{3} & E_4& E_{2}D_{2} & D_1 \\
			C_1 & A_1 & 0&0&0&0\\
			C_{3}F_{3} & 0 & A_{3} & 0 &0&0\\
			C_{4}F_{4} & 0 & 0 & A_{4} &0&0\\
			F_{2} & 0 & 0 & 0& B_{2}& 0 \\
			F_{1} & 0 & 0 & 0& 0& B_{1}\\
			\end{array}
			\right)=r\left(
			\begin{array}{ccc}
			E_{1} & E_{3} & E_4\\
			A_{1} & 0 &0\\
			0 & A_{3} &0\\
			0 & 0 & A_{4}
			\end{array}
			\right)+r\begin{pmatrix}
			F_2& B_2 &0\\
			F_1&0&B_1
			\end{pmatrix},
			\end{aligned}\\
			&\begin{aligned}\label{3.10}
			&r\begin{pmatrix}
			C_c & E_{4} & E_1& E_{2}D_{2} & E_{3}D_{3} & D_1 \\
			F_{2} & 0 & 0& B_{2} & 0 &0\\
			F_{3} & 0 & 0 & 0& B_{3} &0 \\
			F_{1} & 0 & 0 & 0& 0 &B_{1} \\
			C_{4}F_{4} & A_{4} & 0 & 0 &0&0\\
			C_{1} & 0 & A_{1} & 0 &0&0\\
			\end{pmatrix}=r\begin{pmatrix}
			F_{2} & B_{2} & 0 &0\\
			F_{3} & 0 & B_{3} &0\\
			F_{1} & 0 & 0 & B_{1}\\
			\end{pmatrix}+r
			\begin{pmatrix}
			E_{4} &E_1 \\
			A_{4}&0 \\
			0&A_1
			\end{pmatrix},
			\end{aligned}\\
			&\begin{aligned}\label{3.11}
			&r\left(
			\begin{array}{cccccc}
			C_c & E_{1} & E_{2} & E_3& E_{4}D_{4} & D_1 \\
			C_1 & A_1 & 0&0&0&0\\
			C_{2}F_{2} & 0 & A_{2} & 0 &0&0\\
			C_{3}F_{3} & 0 & 0 & A_{3} &0&0\\
			F_{4} & 0 & 0 & 0& B_{4}& 0 \\
			F_{1} & 0 & 0 & 0& 0& B_{1}\\
			\end{array}
			\right)=r\left(
			\begin{array}{ccc}
			E_{1} & E_{2} & E_3\\
			A_{1} & 0 &0\\
			0 & A_{2} &0\\
			0 & 0 & A_{3}
			\end{array}
			\right)+r\begin{pmatrix}
			F_4& B_4 &0\\
			F_1&0&B_1
			\end{pmatrix},
			\end{aligned}\\
			&\begin{aligned}\label{3.12}
			& r\begin{pmatrix}
			C_c & E_{2} & E_1& E_{3}D_{3} & E_{4}D_{4} & D_1 \\
			F_{3} & 0 & 0& B_{3} & 0 &0\\
			F_{4} & 0 & 0 & 0& B_{4} &0 \\
			F_{1} & 0 & 0 & 0& 0 &B_{1} \\
			C_{2}F_{2} & A_{2} & 0 & 0 &0&0\\
			C_{1} & 0 & A_{1} & 0 &0&0\\
			\end{pmatrix}=r \begin{pmatrix}
			F_{3} & B_{3} & 0 &0\\
			F_{4} & 0 & B_{4} &0\\
			F_{1} & 0 & 0 & B_{1}\\
			\end{pmatrix}+r\begin{pmatrix}
			E_{2} &E_1 \\
			A_{2 }&0 \\
			0&A_1
			\end{pmatrix},
			\end{aligned}\\
			&\begin{aligned}\label{3.13}
			&r\begin{pmatrix}
			C_c & E_{3} & E_1& E_{2}D_{2} & E_{4}D_{4} & D_1 \\
			F_{2} & 0 & 0& B_{2} & 0 &0\\
			F_{4} & 0 & 0 & 0& B_{4} &0 \\
			F_{1} & 0 & 0 & 0& 0 &B_{1} \\
			C_{3}F_{3} & A_{3} & 0 & 0 &0&0\\
			C_{1} & 0 & A_{1} & 0 &0&0\\
			\end{pmatrix}=r\begin{pmatrix}
			F_{2} & B_{2} & 0 &0\\
			F_{4} & 0 & B_{4} &0\\
			F_{1} & 0 & 0 & B_{1}\\
			\end{pmatrix}+r\begin{pmatrix}
			E_{3} &E_1 \\
			A_{3}&0 \\
			0&A_1
			\end{pmatrix},
			\end{aligned}\\
			&\begin{aligned}\label{3.14}
			&r\begin{pmatrix}
			C_c & E_{1} & E_4D_4& E_{2}D_{2} & E_{3}D_{3} & D_1 \\
			F_{4} & 0 & B_{4}& 0 & 0 &0\\
			F_{2} & 0 & 0 & B_{2}& 0 &0 \\
			F_{3} & 0 & 0 & 0& B_{3} &0 \\
			F_{1} & 0 & 0 & 0 &0& B_{1}\\
			C_{1} & A_{1} & 0 & 0 &0&0\\
			\end{pmatrix}=r\begin{pmatrix}
			F_{4} & B_{4} & 0 &0&0\\
			F_{2} & 0 & B_{2} &0&0\\
			F_{3} & 0 & 0 & B_{3}&0\\
			F_{1} & 0 & 0 & 0 & B_{1}\\
			\end{pmatrix}+r\begin{pmatrix}
			E_1 \\
			A_1
			\end{pmatrix},
			\end{aligned}
			\end{align}
			\begin{small}
				\begin{align}
				\begin{aligned}\label{3.15}
				&r\left(
				\begin{array}{cccccccccccc}
				C_c & E_2 & E_{1} & 0 & 0&0& E_{4} & E_{3}D_{3} & D_1& 0 &0& E_{4}D_{4} \\
				F_3 & 0 & 0 & 0 & 0&0& 0 & B_{3} & 0& 0 &0& 0 \\
				F_1 & 0 & 0 & 0 & 0&0& 0 & 0 & B_{1}& 0 &0& 0 \\
				0&0&0 & C_c & E_3 & E_{1} & E_{4} & 0&  0 & E_{2}D_{2} & D_1& 0 \\
				0 & 0 & 0 & F_2 & 0&0& 0 & 0 & 0& B_{2} &0& 0 \\
				0 & 0 & 0 & F_1 & 0&0& 0 & 0 & 0& 0 &B_{1}& 0 \\
				F_4 & 0 & 0 & -F_4 & 0&0& 0 & 0 & 0& 0 &0& B_{4} \\
				C_2F_2 & A_2 & 0 & 0 & 0&0& 0 & 0 & 0& 0 &0& 0 \\
				C_1 & 0 & A_1 & 0 & 0&0& 0 & 0 & 0& 0 &0& 0 \\
				0 & 0 & 0 & C_3F_3 & A_3&0& 0 & 0 & 0& 0 &0& 0 \\
				0 & 0 & 0 & C_1 & 0&A_1& 0 & 0 & 0& 0 &0& 0 \\
				0 & 0 & 0 & C_4F_4 & 0&0& A_4 & 0 & 0& 0 &0& 0 \\
				\end{array}
				\right)\\
				&=r\left(
				\begin{array}{ccccccc}
				F_{3} & 0 & B_{3} & 0 & 0 &0 &0\\
				F_{1} & 0 & 0 & B_{1} & 0 &0 &0\\
				0 & F_{2} & 0 & 0 & B_{2} &0 &0\\
				0 & F_{1} & 0 & 0 & 0 &B_{1} &0\\
				F_{4} & F_{4} & 0 & 0 &0&0&B_{4} \\
				\end{array}
				\right)+r\left(
				\begin{array}{ccccc}
				E_2& E_{1} & 0 &0 & E_{4} \\
				0 & 0& E_{3} & E_{1}& E_4 \\
				A_{2} & 0 & 0 &0&0\\
				0 & A_{1} & 0 &0&0\\
				0 & 0 & A_{3} &0&0\\
				0 & 0 & 0 &A_{1}&0\\
				0 & 0 & 0 &0&A_{4}
				\end{array}
				\right).
				\end{aligned}
				\end{align}
			\end{small}
			In this case, the general solution to system \eqref{eq1} is
			\begin{equation}\label{solution}
			\begin{aligned}
			&U=A_{1}^{\dagger}C_{1}+L_{A_1}S_1,\ V=D_{1}B_{1}^{\dagger}+S_2R_{A_1},\ X=A_{2}^{\dagger}C_{2}+L_{A_{2}}D_{2}B_{2}^{\dagger}+L_{A_{2}}U_{1}R_{B_{2}},\\
			&Y=A_{3}^{\dagger}C_{3}+L_{A_{3}}D_{3}B_{3}^{\dagger}+L_{A_{3}}U_{2}R_{B_{3}},\ Z=A_{4}^{\dagger}C_{4}+L_{A_{4}}D_{4}B_{4}^{\dagger}+L_{A_{4}}U_{3}R_{B_{4}},
			\end{aligned}
			\end{equation}
			where
			\begin{align*}
			&S_{1}=A_{11}^{\dagger}(T_1-A_{22}XB_{22}-A_{33}YB_{33}-A_{44}ZB_{44})-A_{11}^{\dagger}W_{11}B_{11}+L_{A_{11}}W_{12},\\
			&S_{2}=R_{A_{11}}(T_1-A_{22}XB_{22}-A_{33}YB_{33}-A_{44}ZB_{44})B_{11}^{\dagger}+A_{11}A_{11}^{\dagger}W_{11}+W_{13}R_{B_{11}},\\
	&U_{1}=A_{12}^{\dagger}TB_{21}^{\dagger}-A_{12}^{\dagger}A_{13}M_{1}^{\dagger}TB_{21}^{\dagger}-A_{12}^{\dagger}S_{1}A_{13}^{\dagger}TN_{1}^{\dagger}B_{31}B_{21}^{\dagger}\\
			&-A_{12}^{\dagger}S_{1}U_{4}R_{N_{1}}B_{31}B_{21}^{\dagger}+L_{A_{21}}U_{5}+U_{6}R_{B_{21}},\\
			&U_{2}=M_{1}^{\dagger}TB_{31}^{\dagger}+S_{1}^{\dagger}S_{1}A_{22}^{\dagger}TN_{1}^{\dagger}+L_{M_{1}}L_{S_{1}}U_{7}+U_{8}R_{B_{22}}+L_{M_{1}}U_{4}R_{N_{1}},\\
			&U_{3}=F_{10}+L_{G_{2}}V_{1}+V_{2}R_{H_{1}}+L_{G_{1}}V_{3}R_{H_{2}}, \ or
U_{3}=F_{20}-L_{G_{4}}W_{1}-W_{2}R_{H_{3}}-L_{G_{3}}W_{3}R_{H_{4}},\\
			&V_{1}=(I_{m},\ 0)\left[C_{11}^{\dagger}(F-C_{22}V_{3}D_{22}-C_{33}W_{3}D_{33})\right]-(I_{m},\ 0)\left[C_{11}^{\dagger}U_{11}D_{11}-L_{C_{11}}U_{12}\right],\\
			&W_{1}=(0,\ I_{m})\left[C_{11}^{\dagger}(F-C_{22}V_{3}D_{22}-C_{33}W_{3}D_{33})\right]-(0,\ I_{m})\left[C_{11}^{\dagger}U_{11}D_{11}-L_{C_{11}}U_{12}\right],\\
			&W_{2}=\left[R_{C_{11}}(F-C_{22}V_{3}D_{22}-C_{33}W_{3}D_{33})D_{11}^{\dagger}\right]\left(
			\begin{array}{c}
			0 \\
			I_{n} \\
			\end{array}
			\right)+\left[C_{11}C_{11}^{\dagger}U_{11}+U_{21}R_{D_{11}}\right]\left(
			\begin{array}{c}
			0 \\
			I_{n} \\
			\end{array}
			\right),\\
\end{align*}
			\begin{align*}
			&V_{2}=\left[R_{C_{11}}(F-C_{22}V_{3}D_{22}-C_{33}W_{3}D_{33})D_{11}^{\dagger}\right]
			\left(                                                                                                                                                                                                  \begin{array}{c}
			I_{n} \\
			0 \\
			\end{array}
			\right)+\left[C_{11}C_{11}^{\dagger}U_{11}+U_{21}R_{D_{11}}\right]
			\left(                                                                                                                                                                                                  \begin{array}{c}
			I_{n} \\
			0 \\
			\end{array}
			\right),\\
			&V_{3}=E_{11}^{\dagger}FE_{33}^{\dagger}-E_{11}^{\dagger}E_{22}M^{\dagger}FE_{33}^{\dagger}-E_{11}^{\dagger}SE_{22}^{\dagger}FN^{\dagger}E_{44}E_{33}^{\dagger}\\
			&-E_{11}^{\dagger}SU_{31}R_{N}E_{44}E_{33}^{\dagger}+L_{E_{11}}U_{32}+U_{33}R_{E_{33}},\\
			&W_{3}=M^{\dagger}FE_{44}^{\dagger}+S^{\dagger}SE_{22}^{\dagger}FN^{\dagger}+L_{M}L_{S}U_{41}+L_{M}U_{31}R_{N}+U_{42}R_{E_{44}},
			\end{align*}
			where $T=T_{1}-A_{33}U_{3}B_{33}$, $U_{j} (j=\overline{4,6})$, $U_{i1} (i=\overline{1,4})$, $U_{12}$, $U_{32}$, $U_{33}$ and $U_{42}$ are arbitrary matrices with appropriate shapes over $\mathbb{H}$. $m$ is the column number of $A_4$ and $n$ is the row number of $B_4$.
		\end{theorem}
		\begin{proof}
			$(1)\Leftrightarrow (2)$ It is clear that the system of matrix equations \eqref{eq1} is solvable if and only if both
			
			\begin{equation}\label{eq315}
			\begin{aligned}
			&A_{1}U=C_1,\ VB_1=D_1,\\
			&A_{2}X=C_{2},\ XB_{2}=D_{2},\\
			&A_{3}Y=C_{3},\ YB_{3}=D_{3},\\
			&A_{4}Z=C_{4},\ ZB_{4}=D_{4}
			\end{aligned}
			\end{equation}
			and
			\begin{equation}\label{eq316}
			\begin{aligned}
			E_1U+VF_1+E_{2}XF_{2}+E_{3}YF_{3}+E_{4}ZF_{4}=C_c
			\end{aligned}
			\end{equation}
			are solvable. It follows from Lemma \ref{lemma2.2}, Lemma \ref{lemma2.3}, and Lemma \ref{lem2.2} that the system of matrix equations \eqref{eq315} has a solution if and only if \eqref{eqnew} holds and
			\begin{equation}\label{eq317}
			\begin{aligned}
			&R_{A_{i}}C_{i}=0,\ B_{i}L_{D_{i}}=0\ (i=\overline{1,4}).
			\end{aligned}
			\end{equation}
			In this case,	the general solution of equations \eqref{eq315} can be expressed as
			\begin{equation}\label{eq318}
			\begin{aligned}
			&U=A_{1}^{\dagger}C_{1}+L_{A_1}S_1,\ V=D_{1}B_{1}^{\dagger}+S_2R_{A_1},\\
			&X=A_{2}^{\dagger}C_{2}+L_{A_{2}}D_{2}B_{2}^{\dagger}+L_{A_{2}}U_{1}R_{B_{2}},\\
			&Y=A_{3}^{\dagger}C_{3}+L_{A_{3}}D_{3}B_{3}^{\dagger}+L_{A_{3}}U_{2}R_{B_{3}},\\ &Z=A_{4}^{\dagger}C_{4}+L_{A_{4}}D_{4}B_{4}^{\dagger}+L_{A_{4}}U_{3}R_{B_{4}}.
			\end{aligned}
			\end{equation}
			
			By substituting $U,V, X,Y,Z$ from \eqref{eq318} into \eqref{eq316} yields	
			\begin{equation}\label{eq319}
			\begin{aligned}
			A_{11}S_1+S_2B_{11}+A_{22}U_{1}B_{22}+A_{33}U_{2}B_{33}+A_{44}U_{3}B_{44 }=T_{1},
			\end{aligned}
			\end{equation}
			where $A_{ii}$, $B_{ii}(i=\overline{1,4})$, and $T_{1}$ are defined by \eqref{3.1}. By Lemma \ref{lem2.3}, we obtain that equation \eqref{eq319} has a solution if and only if
			\begin{equation}\label{eqc1}
			\begin{aligned}
			&R_{G_{i}}L_{i}=0,\ L_{i}L_{H_{i}}=0\ (i=\overline{1,4}),\ R_{E_{22}}EL_{E_{33}}=0.
			\end{aligned}
			\end{equation}
			Under these conditions, the general solution to the matrix equation \eqref{eq319} can be expressed as
			\begin{align*}
			&S_{1}=A_{11}^{\dagger}(T_1-A_{22}XB_{22}-A_{33}YB_{33}-A_{44}ZB_{44})-A_{11}^{\dagger}W_{11}B_{11}+L_{A_{11}}W_{12},\\
			&S_{2}=R_{A_{11}}(T_1-A_{22}XB_{22}-A_{33}YB_{33}-A_{44}ZB_{44})B_{11}^{\dagger}+A_{11}A_{11}^{\dagger}W_{11}+W_{13}R_{B_{11}},\\
	&U_{1}=A_{12}^{\dagger}TB_{21}^{\dagger}-A_{12}^{\dagger}A_{13}M_{1}^{\dagger}TB_{21}^{\dagger}-A_{12}^{\dagger}S_{1}A_{13}^{\dagger}TN_{1}^{\dagger}B_{31}B_{21}^{\dagger}\\
			&-A_{12}^{\dagger}S_{1}U_{4}R_{N_{1}}B_{31}B_{21}^{\dagger}+L_{A_{21}}U_{5}+U_{6}R_{B_{21}},\\
			&U_{2}=M_{1}^{\dagger}TB_{31}^{\dagger}+S_{1}^{\dagger}S_{1}A_{22}^{\dagger}TN_{1}^{\dagger}+L_{M_{1}}L_{S_{1}}U_{7}+U_{8}R_{B_{22}}+L_{M_{1}}U_{4}R_{N_{1}},\\
			&U_{3}=F_{10}+L_{G_{2}}V_{1}+V_{2}R_{H_{1}}+L_{G_{1}}V_{3}R_{H_{2}}, \ or
U_{3}=F_{20}-L_{G_{4}}W_{1}-W_{2}R_{H_{3}}-L_{G_{3}}W_{3}R_{H_{4}},\\
			&V_{1}=(I_{m},\ 0)\left[C_{11}^{\dagger}(F-C_{22}V_{3}D_{22}-C_{33}W_{3}D_{33})\right]-(I_{m},\ 0)\left[C_{11}^{\dagger}U_{11}D_{11}-L_{C_{11}}U_{12}\right],\\
			&W_{1}=(0,\ I_{m})\left[C_{11}^{\dagger}(F-C_{22}V_{3}D_{22}-C_{33}W_{3}D_{33})\right]-(0,\ I_{m})\left[C_{11}^{\dagger}U_{11}D_{11}-L_{C_{11}}U_{12}\right],\\
			&W_{2}=\left[R_{C_{11}}(F-C_{22}V_{3}D_{22}-C_{33}W_{3}D_{33})D_{11}^{\dagger}\right]\left(
			\begin{array}{c}
			0 \\
			I_{n} \\
			\end{array}
			\right)+\left[C_{11}C_{11}^{\dagger}U_{11}+U_{21}R_{D_{11}}\right]\left(
			\begin{array}{c}
			0 \\
			I_{n} \\
			\end{array}
			\right),\\
			&V_{2}=\left[R_{C_{11}}(F-C_{22}V_{3}D_{22}-C_{33}W_{3}D_{33})D_{11}^{\dagger}\right]
			\left(                                                                                                                                                                                                  \begin{array}{c}
			I_{n} \\
			0 \\
			\end{array}
			\right)+\left[C_{11}C_{11}^{\dagger}U_{11}+U_{21}R_{D_{11}}\right]
			\left(                                                                                                                                                                                                  \begin{array}{c}
			I_{n} \\
			0 \\
			\end{array}
			\right),\\
			&V_{3}=E_{11}^{\dagger}FE_{33}^{\dagger}-E_{11}^{\dagger}E_{22}M^{\dagger}FE_{33}^{\dagger}-E_{11}^{\dagger}SE_{22}^{\dagger}FN^{\dagger}E_{44}E_{33}^{\dagger}\\
			&-E_{11}^{\dagger}SU_{31}R_{N}E_{44}E_{33}^{\dagger}+L_{E_{11}}U_{32}+U_{33}R_{E_{33}},\\
			&W_{3}=M^{\dagger}FE_{44}^{\dagger}+S^{\dagger}SE_{22}^{\dagger}FN^{\dagger}+L_{M}L_{S}U_{41}+L_{M}U_{31}R_{N}+U_{42}R_{E_{44}},
			\end{align*}
			where $T=T_{1}-A_{33}U_{3}B_{33}$, $U_{j} (j=\overline{4,6})$, $U_{i1} (i=\overline{1,4})$, $U_{12}$, $U_{32}$, $U_{33}$ and $U_{42}$ are arbitrary matrices with appropriate shapes over $\mathbb{H}$. $m$ is the column number of $A_4$ and $n$ is the row number of $B_4$.

			 To sum up, both the equations \eqref{eq315} and the equation \eqref{eq316} are solvable if and only if conditions \eqref{eqnew}, \eqref{eq317} and \eqref{eqc1} hold, i.e, the system of matrix equations \eqref{eq1} has a solution if and only if \eqref{eqnew} and \eqref{eqnewcondition} hold. Under these conditions, the general solution to equations \eqref{eq1} can be expressed as \eqref{solution}.
			
			$(2)\Leftrightarrow (3)$ We first show that $\eqref{eq317}\Leftrightarrow \eqref{3.6}$. According to Lemma \ref{lem2}, it follows that
			\begin{equation}\label{eqc2}
			\begin{aligned}
			&R_{A_{i}}C_{i}=0\Leftrightarrow r(R_{A_{i}}C_{i})=0\Leftrightarrow r(C_{i},\ A_{i})=r(A_{i})(i=1,2,3)\Leftrightarrow \eqref{3.6},\\
			&D_{j}L_{B_{j}}=0\Leftrightarrow r(D_{j}L_{B_{j}})=0\Leftrightarrow r\left(
			\begin{array}{c}
			D_{j} \\
			B_{j} \\
			\end{array}
			\right)=r(B_{j})(j=1,2,3)\Leftrightarrow \eqref{3.6}.
			\end{aligned}
			\end{equation}
			It follows from \eqref{eqc2} that $\eqref{eq317}\Leftrightarrow \eqref{3.6}$.
			
			We now turn to show that \eqref{eqc1} holds if and only if \eqref{3.7} to \eqref{3.15} hold. By Lemma \ref{lem2.3}, \eqref{eqc1} is equivalent to

			\begin{align}
			&\begin{aligned}\label{eq320}
			&r\left(
			\begin{array}{ccccc}
			T_1 & A_{22} & A_{33} & A_{44} & A_{11} \\
			B_{11} & 0 & 0 & 0 & 0 \\
			\end{array}
			\right)=r(B_{11})+r(A_{22},\ A_{33},\ A_{44},\ A_{11}),
			\end{aligned}\\
			&\begin{aligned}\label{eq32}	
			&r\begin{pmatrix}
			T_1 & A_{22} & A_{44} & A_{11} \\
			B_{33} & 0 & 0 & 0 \\
			B_{11} & 0 & 0 & 0 \\
			\end{pmatrix}=r(A_{22},\ A_{44},\ A_{11})+r\begin{pmatrix}
			B_{33} \\
			B_{11} \\
			\end{pmatrix},
			\end{aligned}\\
			&\begin{aligned}\label{eq322}
			&r\begin{pmatrix}
			T_1 & A_{33} & A_{44} & A_{11} \\
			B_{22} & 0 & 0 & 0 \\
			B_{11} & 0 & 0 & 0 \\
			\end{pmatrix}=r(A_{33},\ A_{44},\ A_{11})+r
			\begin{pmatrix}
			B_{22} \\
			B_{11} \\
			\end{pmatrix},
			\end{aligned}\\
			&\begin{aligned}\label{eq323}
			&r
			\begin{pmatrix}
			T_1 & A_{44} & A_{11} \\
			B_{22} & 0 & 0 \\
			B_{33} & 0 & 0 \\
			B_{11} & 0 & 0 \\
			\end{pmatrix}=r
			\begin{pmatrix}
			B_{22} \\
			B_{33} \\
			B_{11} \\
			\end{pmatrix}
			+r(A_{44},\ A_{11}),
			\end{aligned}\\
			&\begin{aligned}\label{eq324}
			&r
			\begin{pmatrix}
			T_1 & A_{22} & A_{33} & A_{11} \\
			B_{44} & 0 & 0 & 0 \\
			B_{11} & 0 & 0 & 0 \\
			\end{pmatrix}=r(A_{22},\ A_{33},\ A_{11})+r
			\begin{pmatrix}
			B_{44} \\
			B_{11} \\
			\end{pmatrix},
			\end{aligned}\\
			&\begin{aligned}\label{eq325}
			&r\left(
			\begin{array}{ccc}
			T_1 & A_{22} & A_{11} \\
			B_{33} & 0 & 0 \\
			B_{44} & 0 & 0 \\
			B_{11} & 0 & 0 \\
			\end{array}
			\right)=r\left(
			\begin{array}{c}
			B_{33} \\
			B_{44} \\
			B_{11} \\
			\end{array}
			\right)+r(A_{22},\ A_{11}),
			\end{aligned}\\
			&\begin{aligned}\label{eq326}
			&r\left(
			\begin{array}{ccc}
			T_1 & A_{33} & A_{11} \\
			B_{22} & 0 & 0 \\
			B_{44} & 0 & 0 \\
			B_{11} & 0 & 0 \\
			\end{array}
			\right)=r\left(
			\begin{array}{c}
			B_{22} \\
			B_{44} \\
			B_{11} \\
			\end{array}
			\right)+r(A_{33},\ A_{11}),
			\end{aligned}\\
			&\begin{aligned}\label{eq327}	&r\left(
			\begin{array}{cc}
			T_1 & A_{11} \\
			B_{22} & 0 \\
			B_{33} & 0 \\
			B_{44} & 0 \\
			B_{11} & 0 \\
			\end{array}
			\right)=r\left(
			\begin{array}{c}
			B_{22} \\
			B_{33} \\
			B_{44} \\
			B_{11} \\
			\end{array}
			\right)+r(A_{11}),
			\end{aligned}\\
			&\begin{aligned} \nonumber
			&r\left(
			\begin{array}{ccccccc}
			T_1 & A_{22} & A_{11} & 0 & 0 & 0 & A_{44} \\
			B_{33} & 0 & 0 & 0 & 0 & 0 & 0 \\
			B_{11} & 0 & 0 & 0 & 0 & 0 & 0 \\
			0 & 0 & 0 & -T_1 & A_{33} & A_{11} & A_{44} \\
			0 & 0 & 0 & B_{22} & 0 & 0 & 0 \\
			0 & 0 & 0 & B_{11} & 0 & 0 & 0 \\
			B_{44} & 0 & 0 & B_{44} & 0 & 0 & 0 \\
			\end{array}
			\right)\\
			\end{aligned}
\end{align}
			\begin{align}
			&\begin{aligned}	\label{eq328}
			&=r\left(
			\begin{array}{cc}
			B_{33} & 0 \\
			B_{11} & 0 \\
			0 & B_{22} \\
			0 & B_{11} \\
			B_{44} & B_{44} \\
			\end{array}
			\right)+r\left(
			\begin{array}{ccccc}
			A_{22} & A_{11} & 0 & 0 & A_{44} \\
			0 & 0 & A_{33} & A_{11} & A_{44} \\
			\end{array}
			\right),
			\end{aligned}
			\end{align}
			respectively. Hence, we only show that
			\begin{align*}
			(19+i)\Leftrightarrow  (36+i) \ (i=\overline{1,9}),
			\end{align*}
respectively. When we show that \eqref{eqc1} holds if and only if \eqref{3.7} to \eqref{3.15} hold, respectively. It is easy to know that there exist the $U_0, V_0$, $X_{0}, Y_{0}$ and $Z_{0}$ of the equations \eqref{eq1} such that
			\begin{equation}\label{eqq1}
			\begin{aligned}
			&A_{1}U_{0}=C_{1},\ V_{0}B_{1}=D_{1},\\
			&A_{2}X_{0}=C_{2},\ X_{0}B_{2}=D_{2},\\
			&A_{3}Y_{0}=C_{3},\ Y_{0}B_{3}=D_{3},\\
			&A_{4}Z_{0}=C_{4},\ Z_{0}B_{4}=D_{4},
			\end{aligned}
			\end{equation}
			where
			\begin{align*}
			&	U_0=A_{1}^{\dagger}C_1,\ V_0=D_{1}B_{1}^{\dagger},\
			X_{0}=A_{1}^{\dagger}C_{1}+L_{A_{1}}D_{1}B_{1}^{\dagger},\\ &Y_{0}=A_{2}^{\dagger}C_{2}+L_{A_{2}}D_{2}B_{2}^{\dagger},\ Z_{0}=A_{3}^{\dagger}C_{3}+L_{A_{3}}D_{3}B_{3}^{\dagger},
			\end{align*}
			 It follows from Lemma \ref{lem2}, \eqref{eqq1} and elementary transformations that
			\begin{align*}
			&\eqref{eq320}\Leftrightarrow r(T_1,\ E_{1}L_{A_{1}},\ E_{2}L_{A_{2}},\ E_{3}L_{A_{3}})=r(E_{1}L_{A_{1}},\ E_{2}L_{A_{2}},\ E_{3}L_{A_{3}})\\
			&\Leftrightarrow r\left(
			\begin{array}{cccc}
			C & E_{1} & E_{2} & E_{3} \\
			C_{1}F_{1} & A_{1} & 0 & 0 \\
			C_{2}F_{2} & 0 & A_{2} & 0 \\
			C_{3}F_{3} & 0 & 0 & A_{3} \\
			\end{array}
			\right)=r\left(
			\begin{array}{ccc}
			E_{1} & E_{2} & E_{3} \\
			A_{1} & 0 & 0 \\
			0 & A_{2} & 0 \\
			0 & 0 & A_{3} \\
			\end{array}
			\right)\Leftrightarrow \eqref{3.7}.
			\end{align*}
Similarly, we can show that $\eqref{eq32}\Leftrightarrow \eqref{3.8},\ \eqref{eq322}\Leftrightarrow \eqref{3.9}, \eqref{eq323}\Leftrightarrow \eqref{3.10},
			\eqref{eq324}\Leftrightarrow \eqref{3.11},
			\eqref{eq325}\Leftrightarrow \eqref{3.12}, \eqref{eq326}\Leftrightarrow \eqref{3.13},
			\eqref{eq327}\Leftrightarrow \eqref{3.14},$
			\begin{align*}	
			&\eqref{eq328} \Leftrightarrow\\
			& r\left(
			\begin{array}{ccccccc}
			T_1 & E_{2}L_{A_2} & E_{1}L_{A_1} & 0 & 0 & 0 & E_{4}L_{A_4} \\
			R_{B_3}F_3 & 0 & 0 & 0 & 0 & 0 & 0 \\
			R_{B_1}F_1 & 0 & 0 & 0 & 0 & 0 & 0 \\
			0 & 0 & 0 & -T_1 & E_{3}L_{A_3} & E_{1}L_{A_1} & E_{4}L_{A_4} \\
			0 & 0 & 0 & R_{B_2}F_2 & 0 & 0 & 0 \\
			0 & 0 & 0 & R_{B_1}F_1 & 0 & 0 & 0 \\
			R_{B_4}F_4 & 0 & 0 & R_{B_4}F_4 & 0 & 0 & 0 \\
			\end{array}
			\right)
\end{align*}
\begin{align*}	
&=r\left(
			\begin{array}{cc}
			R_{B_3}F_3 & 0 \\
			R_{B_1}F_1 & 0 \\
			0 & R_{B_2}F_2 \\
			0 & R_{B_1}F_1 \\
			R_{B_4}F_4 & R_{B_4}F_4 \\
			\end{array}
			\right)+r\left(
			\begin{array}{ccccc}
			E_{2}L_{A_2} & E_{1}L_{A_1} & 0 & 0 & E_{4}L_{A_4} \\
			0 & 0 & E_{3}L_{A_3} & E_{1}L_{A_1} & E_{4}L_{A_4} \\
			\end{array}
			\right)\\
&\Leftrightarrow r\left(
\begin{array}{cccccccccccc}
C_c & E_2 & E_{1} & 0 & 0&0& E_{4} & E_{3}D_{3} & D_1& 0 &0& E_{4}D_{4} \\
F_3 & 0 & 0 & 0 & 0&0& 0 & B_{3} & 0& 0 &0& 0 \\
F_1 & 0 & 0 & 0 & 0&0& 0 & 0 & B_{1}& 0 &0& 0 \\
0&0&0 & C_c & E_3 & E_{1} & E_{4} & 0&  0 & E_{2}D_{2} & D_1& 0 \\
0 & 0 & 0 & F_2 & 0&0& 0 & 0 & 0& B_{2} &0& 0 \\
0 & 0 & 0 & F_1 & 0&0& 0 & 0 & 0& 0 &B_{1}& 0 \\
F_4 & 0 & 0 & -F_4 & 0&0& 0 & 0 & 0& 0 &0& B_{4} \\
C_2F_2 & A_2 & 0 & 0 & 0&0& 0 & 0 & 0& 0 &0& 0 \\
C_1 & 0 & A_1 & 0 & 0&0& 0 & 0 & 0& 0 &0& 0 \\
0 & 0 & 0 & C_3F_3 & A_3&0& 0 & 0 & 0& 0 &0& 0 \\
0 & 0 & 0 & C_1 & 0&A_1& 0 & 0 & 0& 0 &0& 0 \\
0 & 0 & 0 & C_4F_4 & 0&0& A_4 & 0 & 0& 0 &0& 0 \\
\end{array}
\right)\\
&=r\left(
\begin{array}{ccccccc}
F_{3} & 0 & B_{3} & 0 & 0 &0 &0\\
F_{1} & 0 & 0 & B_{1} & 0 &0 &0\\
0 & F_{2} & 0 & 0 & B_{2} &0 &0\\
0 & F_{1} & 0 & 0 & 0 &B_{1} &0\\
F_{4} & F_{4} & 0 & 0 &0&0&B_{4} \\
\end{array}
\right)+r\left(
\begin{array}{ccccc}
E_2& E_{1} & 0 &0 & E_{4} \\
0 & 0& E_{3} & E_{1}& E_4 \\
A_{2} & 0 & 0 &0&0\\
0 & A_{1} & 0 &0&0\\
0 & 0 & A_{3} &0&0\\
0 & 0 & 0 &A_{1}&0\\
0 & 0 & 0 &0&A_{4}
\end{array}
\right)\Leftrightarrow \eqref{3.15}.
\end{align*}
			We have thus proved the theorem.
\end{proof}

\textbf{Remark 3.2.} Chu et al. gave potential applications of the maximal and minimal ranks in the discipline of control theory(e.g., \cite{D.L. 1998}, \cite{D.L. 2000}, \cite{D.L. 2009}). We may consider the rank bounds of the general solution of the equation \eqref{eq1}.
		
		 Next, we discuss the special case of \eqref{eq1}. Let $A_{i}$, $B_{i}$, $C_{i}$, $D_{i}$, $E_{i}$, $F_{i}\ (i=\overline{1,3})$ and $C$ be given matrices of appropriate sizes over  $\mathbb{H}$.
		\begin{align*}
		&E_{i}L_{A_{i}}=A_{ii},\ R_{B_{i}}F_{i}=B_{ii}(i=\overline{1,3}),\ M_{1}=R_{A_{11}}A_{22},\ N_{1}=B_{22}L_{B_{11}},\ S_{1}=A_{22}L_{M_{1}},\\
		&  G=R_{M_{1}}R_{A_{11}},\ T_{1}=C-\left[\sum_{i=1}^{3}E_{i}(A_{i}^{\dagger}C_{i}+L_{A_{i}}D_{i}B_{i}^{\dagger})F_{i}\right],\ G_{1}=GA_{33},\ G_{2}=R_{A_{11}}A_{33},\\
		&G_{3}=R_{A_{22}}A_{33},\ G_{4}=A_{33},\ H=L_{B_{11}}L_{N_{1}},\ H_{1}=B_{33},\ L_{1}=GT_{1}, \ H_{2}=B_{33}L_{B_{22}},\\
		&  H_{3}=B_{33}L_{B_{11}},\ H_{4}=B_{33}D,\ L_{2}=R_{A_{11}}T_{1}L_{B_{22}},\ L_{3}=R_{A_{22}}T_{1}L_{B_{11}},\ L_{4}=T_{1}H,\\
		&C_{11}=(L_{G_{2}},\ L_{G_{4}}),\ D_{11}=\left(
		\begin{array}{c}
		R_{H_{1}} \\
		R_{H_{3}} \\
		\end{array}
		\right),\ C_{22}=L_{G_{1}},\  D_{22}=R_{H_{2}},\ C_{33}=L_{G_{3}},\ D_{33}=R_{H_{4}},
\end{align*}
\begin{align*}
		& E_{11}=R_{C_{11}}C_{22},\ E_{22}=R_{C_{11}}C_{33},\ E_{33}=D_{22}L_{D_{11}},\ E_{44}=D_{33}L_{D_{11}},\ M=R_{E_{11}}E_{22},\ N=E_{44}L_{E_{33}},\\
		&F=F_{20}-F_{10},\ E=R_{C_{11}}FL_{D_{11}},\ S=E_{22}L_{M},\ F_{11}=G_{2}L_{G_{1}},\\
&G_{5}=L_{2}-G_{2}G_{1}^{\dagger}L_{1}H_{1}^{\dagger}H_{2},\ F_{22}=G_{4}L_{G_{3}},\ G_{6}=L_{4}-G_{4}G_{3}^{\dagger}L_{3}H_{3}^{\dagger}H_{4},\\
		& F_{10}=G_{1}^{\dagger}L_{1}H_{1}^{\dagger}+L_{G_{1}}G_{2}^{\dagger}L_{2}H_{2}^{\dagger},\ F_{20}=G_{3}^{\dagger}L_{3}H_{3}^{\dagger}+L_{G_{3}}G_{4}^{\dagger}L_{4}H_{4}^{\dagger}.
		\end{align*}

		\begin{theorem}\label{them3.2}  The following statements are equivalent:
			
			$\mathrm{(1)}$ system \eqref{eq7-1} has a solution.
			
			$\mathrm{(2)}$
			\begin{align}\label{eqnew1}
			& A_{i}D_{i}=C_{i}B_{i},\ (i=\overline{1,3})
			\end{align}
			
			and
			\begin{align*}
			&R_{A_{i}}C_{i}=0,\ D_{i}L_{B_{i}}=0,\ R_{G_{j}}L_{j}=0,\ L_{j}L_{H_{j}}=0\ (i=\overline{1,3}, j=\overline{1,4}),\ R_{E_{22}}EL_{E_{33}}=0.
			\end{align*}
		
			$\mathrm{(3)}$\ \eqref{eqnew1} holds and for $i=\overline{1,3}$.
			\begin{align*}
			&r(C_{i},\ A_{i})=r(A_{i}),\ r\left(
			\begin{array}{c}
			D_{i} \\
			B_{i} \\
			\end{array}
			\right)=r(B_{i}),\\
			&r\left(
			\begin{array}{cccc}
			C & E_{1} & E_{2} & E_{3} \\
			C_{1}F_{1} & A_{1} & 0 & 0 \\
			C_{2}F_{2} & 0 & A_{2} & 0 \\
			C_{5}F_{3} & 0 & 0 & A_{3} \\
			\end{array}
			\right)=r\left(
			\begin{array}{ccc}
			E_{1} & E_{2} & E_{3} \\
			A_{1} & 0 & 0 \\
			0 & A_{2} & 0 \\
			0 & 0 & A_{3} \\
			\end{array}
			\right),\\
&r\begin{pmatrix}
			C & E_{1} & E_{3} & E_{2}D_{2} \\
			F_{2} & 0 & 0 & B_{2} \\
			C_{1}F_{1} & A_{1} & 0 & 0 \\
			C_{3}F_{3} & 0 & A_{3} & 0 \\
			\end{pmatrix}=r\begin{pmatrix}
			E_{1} & E_{3} \\
			A_{1} & 0 \\
			0 & A_{3} \\
			\end{pmatrix}+r(F_{2},\ B_{2}),\\
			&r\begin{pmatrix}
			C & E_{3} & E_{2} & E_{1}D_{1} \\
			F_{1} & 0 & 0 & B_{1} \\
			C_{3}F_{3} & A_{3} & 0 & 0 \\
			C_{2}F_{2} & 0 & A_{2} & 0 \\
			\end{pmatrix}=r\begin{pmatrix}
			E_{3} & E_{2} \\
			A_{1} & 0 \\
			0 & A_{3} \\
			\end{pmatrix}+r(F_{1},\ B_{1}),\\
			&r\begin{pmatrix}
			C & E_{3} & E_{1}D_{1} & E_{2}D_{2} \\
			F_{1} & 0 & B_{1} & 0 \\
			F_{2} & 0 & 0 & B_{2} \\
			C_{3}F_{3} & A_{3} & 0 & 0 \\
			\end{pmatrix}=r\begin{pmatrix}
			F_{1} & B_{1} & 0 \\
			F_{2} & 0 & B_{2} \\
			\end{pmatrix}+r\begin{pmatrix}
			E_{3} \\
			A_{3} \\
			\end{pmatrix},\\
&r\begin{pmatrix}
			C & E_{1} & E_{2} & E_{3}D_{3} \\
			F_{3} & 0 & 0 & B_{3} \\
			C_{1}F_{1} & A_{1} & 0 & 0 \\
			C_{2}F_{2} & 0 & A_{2} & 0 \\
			\end{pmatrix}=r\begin{pmatrix}
			E_{1} & E_{2} \\
			A_{1} & 0 \\
			0 & A_{2} \\
			\end{pmatrix}+r(F_{3},\ B_{3}),
\end{align*}
			\begin{align*}
			&r\begin{pmatrix}
			C & E_{1} & E_{3}D_{3} & E_{2}D_{2} \\
			F_{3} & 0 & B_{3} & 0 \\
			F_{2} & 0 & 0 & B_{2} \\
			C_{1}F_{1} & A_{1} & 0 & 0 \\
			\end{pmatrix}=r\begin{pmatrix}
			F_{3} & B_{3} & 0 \\
			F_{2} & 0 & B_{2} \\
			\end{pmatrix}+r\begin{pmatrix}
			E_{1} \\
			A_{1} \\
			\end{pmatrix},\\
&r\begin{pmatrix}
			C & E_{2} & E_{1}D_{1} & E_{3}D_{3} \\
			F_{1} & 0 & B_{1} & 0 \\
			F_{3} & 0 & 0 & B_{3} \\
			C_{2}F_{2} & A_{2} & 0 & 0 \\
			\end{pmatrix}=r\begin{pmatrix}
			F_{1} & B_{1} & 0 \\
			F_{3} & 0 & B_{3} \\
			\end{pmatrix}+r\begin{pmatrix}
			E_{2} \\
			A_{2} \\
			\end{pmatrix},\\
			&r\begin{pmatrix}
			C & E_{1}D_{1} & E_{2}D_{2} & E_{3}D_{3} \\
			F_{1} & B_{1} & 0 & 0 \\
			F_{2} & 0 & B_{2} & 0 \\
			F_{3} & 0 & 0 & B_{3} \\
			\end{pmatrix}=r\begin{pmatrix}
			F_{1} & B_{1} & 0 & 0 \\
			F_{2} & 0 & B_{2} & 0 \\
			F_{3} & 0 & 0 & B_{3} \\
			\end{pmatrix},\\
			&r\begin{pmatrix}
			C & 0 & E_{1} & 0 & E_{3} & E_{2}D_{2} & 0 & E_{3}D_{3} \\
			0 & -C & 0 & E_{2} & E_{3} & 0 & -E_{1}D_{1} & 0 \\
			F_{2} & 0 & 0 & 0 & 0 & B_{2} & 0 & 0 \\
			0 & F_{1} & 0 & 0 & 0 & 0 & B_{1} & 0 \\
			F_{3} & F_{3} & 0 & 0 & 0 & 0 & 0 & B_{3} \\
			C_{1}F_{1} & 0 & A_{1} & 0 & 0 & 0 & 0 & 0 \\
			0 & -C_{2}F_{2} & 0 & A_{2} & 0 & 0 & 0 & 0 \\
			0 & -C_{3}F_{3} & 0 & 0 & A_{3} & 0 & 0 & 0 \\
			\end{pmatrix}\\
			&=r\begin{pmatrix}
			F_{2} & 0 & B_{2} & 0 & 0 \\
			0 & F_{1} & 0 & B_{1} & 0 \\
			F_{3} & F_{3} & 0 & 0 & B_{3} \\
			\end{pmatrix}+r\begin{pmatrix}
			E_{1} &0&E_{3}\\
			0 & E_2&E_3\\
			A_1& 0&0\\
			0&A_2&0\\
			0&0&A_3
			\end{pmatrix}.
\end{align*}
In this case, the general solution to system \eqref{eq7-1} is
			\begin{align*}
			&X=A_{1}^{\dagger}C_{1}+L_{A_{1}}D_{1}B_{1}^{\dagger}+L_{A_{1}}U_{1}R_{B_{1}},\ Y=A_{2}^{\dagger}C_{2}+L_{A_{2}}D_{2}B_{2}^{\dagger}+L_{A_{2}}U_{2}R_{B_{2}},\\ &Z=A_{3}^{\dagger}C_{3}+L_{A_{3}}D_{3}B_{3}^{\dagger}+L_{A_{3}}U_{3}R_{B_{3}},
			\end{align*}
			where
			\begin{align*}
&U_{1}=A_{11}^{\dagger}TB_{11}^{\dagger}-A_{11}^{\dagger}A_{22}M_{1}^{\dagger}TB_{11}^{\dagger}-A_{11}^{\dagger}S_{1}A_{22}^{\dagger}TN_{1}^{\dagger}B_{22}B_{11}^{\dagger}-A_{11}^{\dagger}S_{1}U_{4}R_{N_{1}}B_{22}B_{11}^{\dagger}+L_{A_{11}}U_{5}+U_{6}R_{B_{11}},\\
			&U_{2}=M_{1}^{\dagger}TB_{22}^{\dagger}+S_{1}^{\dagger}S_{1}A_{22}^{\dagger}TN_{1}^{\dagger}+L_{M_{1}}L_{S_{1}}U_{7}+U_{8}R_{B_{22}}+L_{M_{1}}U_{4}R_{N_{1}},\\
			&U_{3}=F_{10}+L_{G_{2}}V_{1}+V_{2}R_{H_{1}}+L_{G_{1}}V_{3}R_{H_{2}},\\
			& or \\ &U_{3}=F_{20}-L_{G_{4}}W_{1}-W_{2}R_{H_{3}}-L_{G_{3}}W_{3}R_{H_{4}},\\
			&V_{1}=(I_{m},\ 0)\left[C_{11}^{\dagger}(F-C_{22}V_{3}D_{22}-C_{33}W_{3}D_{33})\right]-(I_{m},\ 0)\left[C_{11}^{\dagger}U_{11}D_{11}-L_{C_{11}}U_{12}\right],\\
				\end{align*}
	\begin{align*}
&W_{1}=(0,\ I_{m})\left[C_{11}^{\dagger}(F-C_{22}V_{3}D_{22}-C_{33}W_{3}D_{33})\right]\\
			&-(0,\ I_{m})\left[C_{11}^{\dagger}U_{11}D_{11}-L_{C_{11}}U_{12}\right],\\
			&W_{2}=\left[R_{C_{11}}(F-C_{22}V_{3}D_{22}-C_{33}W_{3}D_{33})D_{11}^{\dagger}\right]\left(
			\begin{array}{c}
			0 \\
			I_{n} \\
			\end{array}
			\right)\\
			&+\left[C_{11}C_{11}^{\dagger}U_{11}+U_{21}R_{D_{11}}\right]\left(
			\begin{array}{c}
			0 \\
			I_{n} \\
			\end{array}
			\right),\\
			&V_{2}=\left[R_{C_{11}}(F-C_{22}V_{3}D_{22}-C_{33}W_{3}D_{33})D_{11}^{\dagger}\right]
			\left(                                                                                                                                                                                                  \begin{array}{c}
			I_{n} \\
			0 \\
			\end{array}
			\right)\\
			&+\left[C_{11}C_{11}^{\dagger}U_{11}+U_{21}R_{D_{11}}\right]
			\left(                                                                                                                                                                                                  \begin{array}{c}
			I_{n} \\
			0 \\
			\end{array}
			\right),\\
			&V_{3}=E_{11}^{\dagger}FE_{33}^{\dagger}-E_{11}^{\dagger}E_{22}M^{\dagger}FE_{33}^{\dagger}-E_{11}^{\dagger}SE_{22}^{\dagger}FN^{\dagger}E_{44}E_{33}^{\dagger}\\
			&-E_{11}^{\dagger}SU_{31}R_{N}E_{44}E_{33}^{\dagger}+L_{E_{11}}U_{32}+U_{33}R_{E_{33}},\\
			&W_{3}=M^{\dagger}FE_{44}^{\dagger}+S^{\dagger}SE_{22}^{\dagger}FN^{\dagger}+L_{M}L_{S}U_{41}+L_{M}U_{31}R_{N}\\
			&+U_{42}R_{E_{44}},
			\end{align*}
			where $T=T_{1}-A_{33}U_{3}B_{33}$, $U_{j} (j=\overline{4,6})$, $U_{i1} (i=\overline{1,4})$, $U_{12}$, $U_{32}$, $U_{33}$ and $U_{42}$ are arbitrary matrices with appropriate shapes over $\mathbb{H}$. $m$ is the column number of $A_3$ and $n$ is the row number of $B_3$.
		\end{theorem}
	\begin{proof}
		It follows from Theorem \ref{the1} that this holds when $A_1, C_1, B_1, D_1, E_1, F_1$ vanish in Theorem \ref{the1}
	\end{proof}
		Lettiing $A_i,\ B_i,\ C_i,\ D_i (i=\overline{1,3})$, $E_3$ and $F_3$ vanish in Theorem 3.1, it yields to the following result:
		\begin{corollary}\label{cor4.2} Let $C_{3}$, $D_{3}$, $C_{4}$, $D_{4}$ and $E_{1}$ be given matrices with adequate shapes.
			Let $M_{1}=R_{C_{3}}C_{4}, N_{1}=D_{4}L_{D_{3}}, S_{1}=C_{4}L_{M_{1}}$. Then the matrix equation \eqref{eq3} is consistent if and only if the following rank equalities hold:
\begin{small}
			\begin{align*}
			&r(C_{3}\ E_{1}\ C_{4})=r(C_{3}\ C_{4}), r\left(
			\begin{array}{c}
			D_{3} \\
			E_{1} \\
			D_{4} \\
			\end{array}\right)=r\left(
			\begin{array}{c}
			D_{3} \\
			D_{4} \\
			\end{array}
			\right),\\                                      &r
			\begin{pmatrix}
			C_{3} & E_{1} \\
			0 & D_{4} \\
			\end{pmatrix}=r(C_{3})+r(D_{4}),\ r
			\begin{pmatrix}
			D_{3} & 0 \\
			E_{1} & C_{4} \\
			\end{pmatrix}=r(D_{3})+r(C_{4}).
			\end{align*}
			\end{small}
			In this case, the general solution to equation \eqref{eq3} can be expressed as	
\begin{align*}
			&X_{3}=C_{3}^{\dagger}E_{1}D_{3}^{\dagger}-C_{3}^{\dagger}C_{4}M_{1}^{\dagger}E_{1}D_{3}^{\dagger}-C_{3}^{\dagger}S_{1}C_{4}^{\dagger}E_{1}N_{1}^{\dagger}D_{4}D_{3}^{\dagger}\\
			&-C_{3}^{\dagger}S_{1}Y_{11}R_{N_{1}}D_{4}D_{4}^{\dagger}+L_{C_{3}}Y_{12}+Y_{13}R_{D_{3}},\\
			&X_{4}=M_{1}^{\dagger}E_{1}D_{4}^{\dagger}+S_{1}^{\dagger}S_{1}C_{4}^{\dagger}E_{1}N_{1}^{\dagger}+L_{M_{1}}L_{S_{1}}Y_{14}+Y_{15}R_{D_{4}}\\
			&+L_{M_{1}}Y_{11}R_{N_{1}},
			\end{align*}
\end{corollary}

where $Y_{1i} (i=\overline{1,5})$ are any matrices with appropriate sizes over $\mathbb{H}$.
	
		\noindent{\bf Remark 3.3.}\quad The above corollary has the main findings of \cite{J.K. 1980}.
		
		Lettiing $A_3,\ B_3,\ C_3,\ D_3$, $E_3$  and $F_3$ vanish in Theorem 3.1, it yields to the following result:
		\begin{corollary}\label{cor4.3}Let $ A_{i}, B_{i}, C_{i}(i=\overline{1,4})$ and $C_{c}$ be given with appropriate sizes over $\mathbb{H}$. Set
			\begin{align*}
			&A=A_{3}L_{A_{1}},\ B=R_{B_{1}}B_{3},\ C=A_{4}L_{A_{2}},\ D=R_{B_{2}}B_{4},\\
			& M=R_{A}C,\ N=DL_{B},\ S=CL_{M},\ E=C_{c}-A_{3}A_{1}^{\dagger}C_{1}B_{3}\\
			&-AC_{2}B_{1}^{\dagger}B_{3}-A_{4}A_{2}^{\dagger}C_{3}B_{4}-CC_{4}B_{2}^{\dagger}B_{4}.
			\end{align*}
			Then the following statements are equivalent:
			
			$\mathrm{(1)}$  the system of matrix equations \eqref{eq4} is solvable.
			
			$\mathrm{(2)}$
			\begin{align*}
			&A_{1}C_{2}=C_{1}B_{1},\ A_{2}C_{4}=C_{3}B_{2},\ R_{A_{1}}C_{1}=0,\\
			& R_{A_{2}}C_{3}=0,\ C_{2}L_{B_{1}}=0,\ C_{4}L_{B_{2}}=0,\\
			&R_{M}R_{A}E=0,\ R_{A}EL_{D}=0,\ EL_{B}L_{N}=0,\ R_{C}EL_{B}=0.
			\end{align*}
			
			$\mathrm{(3)}$
			\begin{align*}
			& A_{1}C_{2}=C_{1}B_{1},\ A_{2}C_{4}=C_{3}B_{2},\ r(A_{1},\ C_{1})=r(A_{1}),\\
			& r(A_{2},\ C_{3})=r(A_{2}),\ r\left(
			\begin{array}{c}
			C_{2} \\
			B_{1} \\
			\end{array}
			\right)=r(B_{1}),\ r\left(
			\begin{array}{c}
			C_{4} \\
			B_{2} \\
			\end{array}
			\right)=r(B_{2}),\\
			& r\left(
			\begin{array}{ccc}
			A_{1} & 0 & C_{1}B_{3} \\
			A_{3} & A_{4}C_{4} & C_{c} \\
			0 & B_{2} & B_{4} \\
			\end{array}
			\right)=r\left(
			\begin{array}{ccc}
			A_{1} & 0 & 0 \\
			A_{3} & 0 & 0 \\
			0 & B_{2} & B_{4} \\
			\end{array}
			\right),\\
			&r\left(
			\begin{array}{ccc}
			A_{2} & 0 & C_{3}B_{4} \\
			A_{4} & A_{3}C_{2} & C_{c} \\
			0 & B_{1} & B_{3} \\
			\end{array}
			\right)=r\left(
			\begin{array}{ccc}
			A_{2} & 0 & 0 \\
			A_{4} & 0 & 0 \\
			0 & B_{1} & B_{3} \\
			\end{array}
			\right),\\
			&r\left(
			\begin{array}{ccc}
			B_{1} & 0 & B_{3} \\
			0 & B_{2} & B_{4} \\
			A_{3}C_{2} & A_{4}C_{4} & C_{c} \\
			\end{array}
			\right)=r\left(
			\begin{array}{ccc}
			B_{1} & 0 & B_{3} \\
			0 & B_{2} & B_{4} \\
			\end{array}
			\right),\\
			&r\left(
			\begin{array}{ccc}
			C_{1}B_{3} & A_{1} & 0 \\
			C_{3}B_{4} & 0 & A_{2} \\
			C_{c} & A_{3} & A_{4} \\
			\end{array}
			\right)=r\left(
			\begin{array}{cc}
			A_{1} & 0 \\
			0 & A_{2} \\
			A_{3} & A_{4} \\
			\end{array}
			\right).
			\end{align*}
In this case, the general solution to the system  \eqref{eq4} can be expressed as
			\begin{align*}
			&X_{1}=A_{1}^{\dagger}C_{1}+L_{A_{1}}C_{2}B_{1}^{\dagger}+L_{A_{1}}A^{\dagger}EB^{\dagger}R_{B_{1}}-L_{A_{1}}A^{\dagger}CM^{\dagger}R_{A}EB^{\dagger}R_{B_{1}}-L_{A_{1}}A^{\dagger}SC^{\dagger}EL_{B}N^{\dagger}DB^{\dagger}R_{B_{1}}\\
			&-L_{A_{1}}A^{\dagger}SVR_{N}DB^{\dagger}R_{B_{1}}+L_{A_{1}}(L_{A}U+ZR_{B})R_{B_{1}},\\
			&X_{2}=A_{2}^{\dagger}C_{3}+L_{A_{2}}C_{4}B_{2}^{\dagger}+L_{A_{2}}M^{\dagger}R_{A}ED^{\dagger}R_{B_{2}}+L_{A}L_{M}S^{\dagger}SC^{\dagger}EL_{B}N^{\dagger}R_{B_{2}}\\
			&+L_{A_{2}}L_{M}(V-S^{\dagger}SVNN^{\dagger})R_{B_{2}}+L_{A_{2}}WR_{D}R_{B_{2}},
			\end{align*}
			where $U, V, W$ and $Z$ are arbitrary matrices with appropriate sizes over $\mathbb{H}$.
	\end{corollary}
\noindent{\bf Remark 3.4.}\quad The above corollary has the main findings of \cite{Wang 2009}.
		\section{The general solution to equations \eqref{eq7-2} with $\eta$-Hermicity}
		In this section, as an application of equations \eqref{eq1}, we establish some necessary and sufficient conditions for the system of matrix equations \eqref{eq7-2} to have a solution, and derive a formula for its general solution, where $X, Y, Z$ are $\eta$-Hermitian.
		Let $ A_{i}, B_{i}, E_{i}\ (i=\overline{1,4})$ and $C_c$ be given with appropriate sizes over $\mathbb{H}$. Set
		\begin{align*}
		&E_{i}L_{A_{i}}=A_{ii}(i=\overline{1,4}),\ R_{A_{11}}A_{jj}=A_{1j} (j=\overline{2,4}),\ R_{A_{12}}A_{13}=M_{1},\ S_{1}=A_{22}L_{M_{1}},\\
		 &T_{1}=C_c-E_1A_{1}^{\dagger}C_1-C_1^{\eta^{\ast}}(A_{1}^{\eta^{\ast}})^{\dagger}E_{1}^{\eta^{\ast}}-\left[\sum_{i=1}^{3}E_{i}\left(A_{i}^{\dagger}B_{i}+L_{A_{i}}B_{i}^{\eta^{\ast}}(A_{i}^{\eta^{\ast}})^{\dagger}\right)E_{i}^{\eta^{\ast}}\right],\\
		 &T_2=R_{A_{11}}T_1(R_{A_{11}})^{\eta^{\ast}},\ G=R_{M_{1}}R_{A_{12}},\ G_{1}=GA_{14},\ G_{2}=R_{A_{12}}A_{14},\ G_{3}=R_{A_{13}}A_{14},\\ &G_{4}=A_{14},\ L_{1}=GT_{2},\ L_{2}=R_{A_{12}}T_{2}(R_{A_{13}})^{\eta^{\ast}},\\
		&L_{3}=R_{A_{13}}T_{2}(R_{A_{12}})^{\eta^{\ast}},\ L_{4}=T_{2}G^{\eta^{\ast}},\ C_{11}=(L_{G_{2}},\ L_{G_{4}}),\ E_{11}=R_{C_{11}}C_{22},\ C_{22}=L_{G_{1}},\ C_{33}=L_{G_{3}},\\
		&E_{22}=R_{C_{11}}C_{33},\ M=R_{E_{11}}E_{22},\ N=(R_{E_{22}}E_{11})^{\eta^{\ast}},\ F=F_{44}-F_{33},\ E=R_{C_{11}}F(R_{C_{11}})^{\eta^{\ast}},\ S=E_{22}L_{M},\\ &F_{11}=G_{2}L_{G_{1}},\ G_{1}=L_{2}-G_{2}G_{1}^{\dagger}L_{1}(G_{4}^{\eta^{\ast}})^{\dagger}G_{3}^{\eta^{\ast}},\ F_{22}=G_{4}L_{G_{3}},\ G_{2}=L_{4}-G_{4}G_{3}^{\dagger}L_{3}(G_{2}^{\eta^{\ast}})^{\dagger}G_{1}^{\eta^{\ast}},\\
		&F_{1}=G_{1}^{\dagger}G_{1}(G_{4}^{\eta^{\ast}})^{\dagger}+L_{G_{1}}G_{2}^{\dagger}L_{2}(G_{3}^{\eta^{\ast}})^{\dagger},\ F_{2}=G_{3}^{\dagger}L_{3}(G_{2}^{\eta^{\ast}})^{\dagger}+L_{G_{3}}G_{4}^{\dagger}L_{4}(G_{1}^{\eta^{\ast}})^{\dagger}.
		\end{align*}
		Then we have the following theorem.
	\begin{theorem}\label{them4.1} Consider \eqref{eq7-2}. The following statements are equivalent:\\
		$\mathrm{(1)}$ The system of matrix equations \eqref{eq7-2} has a solution.\\
		$\mathrm{(2)}$
		\begin{align*}
		&R_{E_{22}}E(R_{E_{22}})^{\eta^{\ast}}=0,\ R_{A_{i}}B_{i}=0,\ R_{G_{i}}L_{i}=0\ (i=\overline{1,4}).
		\end{align*}	
$\mathrm{(3)}$
		\begin{align*}
		&r(B_i,\ A_i)=r(A_i)\ (i=\overline{1,4}),\\
		&r\begin{pmatrix}
		C_c & E_{1} & E_{2} & E_{3} & E_4 & (C_1)^{\eta^{\ast}}\\
		(E_1)^{\eta^{\ast}} &0 &0 &0 &0 & (A_1)^{\eta^{\ast}}\\
		C_1& A_1& 0& 0&0 &0\\
		C_{2}E_{2}^{\eta^{\ast}} & 0 & A_{2} & 0 & 0& 0\\
		C_{3}E_{3}^{\eta^{\ast}} & 0 & 0 & A_3 & 0& 0\\
		C_{4}E_{4}^{\eta^{\ast}} & 0 & 0 & 0 & A_{4} & 0\\
		\end{pmatrix}=r\begin{pmatrix}
		E_{1} & E_{2} & E_{3} & E_4\\
		A_{1} & 0 & 0 &0\\
		0 & A_{2} & 0 &0\\
		0 & 0 & A_{3} &0\\
		0 & 0 & 0 &A_{4}\\
		\end{pmatrix}+r\begin{pmatrix}
		E_1\\
		A_1
		\end{pmatrix},\\
		&r\begin{pmatrix}
		C_c & E_{4} & E_{2} & E_{1} & (C_3)^{\eta^{\ast}} & (C_1)^{\eta^{\ast}}\\
		(E_3)^{\eta^{\ast}} &0 &0 &0 &(A_3)^{\eta^{\ast}} & 0\\
		(E_1)^{\eta^{\ast}}& 0 & 0& 0&0 &(A_1)^{\eta^{\ast}}\\
		C_{4}E_{4}^{\eta^{\ast}} & A_{4} & 0 & 0 & 0& 0\\
		C_{2}E_{2}^{\eta^{\ast}} & 0 & A_2 & 0 & 0& 0\\
		C_{1} & 0 & 0 & A_{1} & 0 & 0\\
		\end{pmatrix}=r\begin{pmatrix}
		E_{4} & E_{2} & E_1\\
		A_{4} & 0 & 0 \\
		0 & A_{2} & 0 \\
		0 & 0 & A_{1} \\
		\end{pmatrix}+r\begin{pmatrix}
		E_3& E_1\\
		A_3 & 0\\
		0& A_1
		\end{pmatrix},
\end{align*}
\begin{align*}
&r\begin{pmatrix}
		C_c & E_{4} & E_{3} & E_{1} & (C_2)^{\eta^{\ast}} & (C_1)^{\eta^{\ast}}\\
		(E_2)^{\eta^{\ast}} &0 &0 &0 &(A_2)^{\eta^{\ast}} & 0\\
		(E_1)^{\eta^{\ast}}& 0 & 0& 0&0 &(A_1)^{\eta^{\ast}}\\
		C_{4}E_{4}^{\eta^{\ast}} & A_{4} & 0 & 0 & 0& 0\\
		C_{3}E_{3}^{\eta^{\ast}} & 0 & A_3 & 0 & 0& 0\\
		C_{1} & 0 & 0 & A_{1} & 0 & 0\\
		\end{pmatrix}=r\begin{pmatrix}
		E_{4} & E_{3} & E_1\\
		A_{4} & 0 & 0 \\
		0 & A_{3} & 0 \\
		0 & 0 & A_{1} \\
		\end{pmatrix}+r\begin{pmatrix}
		E_2& E_1\\
		A_2 & 0\\
		0& A_1
		\end{pmatrix},\\
		&r\begin{pmatrix}
		C_c & E_{4} & E_{1} & (C_3)^{\eta^{\ast}}& (C_2)^{\eta^{\ast}} & (C_1)^{\eta^{\ast}}\\
		(E_3)^{\eta^{\ast}} &0 &0 &(A_3)^{\eta^{\ast}} &0 & 0\\
		(E_2)^{\eta^{\ast}}& 0 & 0& 0&(A_2)^{\eta^{\ast}} &0\\
		(E_1)^{\eta^{\ast}}& 0 & 0& 0&0 &(A_2)^{\eta^{\ast}}\\
		C_{4}E_{4}^{\eta^{\ast}} & A_{4} & 0 & 0 & 0& 0\\
		C_{1} & 0 & 0 & A_{1} & 0 & 0\\
		\end{pmatrix}=r\begin{pmatrix}
		E_{3} & E_{2} & E_1\\
		A_{3} & 0 & 0 \\
		0 & A_{2} & 0 \\
		0 & 0 & A_{1} \\
		\end{pmatrix}+r\begin{pmatrix}
		E_4& E_1\\
		A_4 & 0\\
		0& A_1
		\end{pmatrix},\\
		&r\left(
		\begin{array}{cccccccccccc}
		C_c & E_2 & E_{1} & 0 & 0&0& E_{4} & P_{3} & C_1^{\eta^{\ast}}& 0 &0& P_{4} \\
		E_3^{\eta^{\ast}} & 0 & 0 & 0 & 0&0& 0 & A_3^{\eta^{\ast}} & 0& 0 &0& 0 \\
		E_1^{\eta^{\ast}} & 0 & 0 & 0 & 0&0& 0 & 0 & A_1^{\eta^{\ast}}& 0 &0& 0 \\
		0&0&0 & C_c & E_3 & E_{1} & E_{4} & 0&  0 & P_{2} & C_1^{\eta^{\ast}}& 0 \\
		0 & 0 & 0 & E_2^{\eta^{\ast}} & 0&0& 0 & 0 & 0& A_2^{\eta^{\ast}} &0& 0 \\
		0 & 0 & 0 & E_1^{\eta^{\ast}} & 0&0& 0 & 0 & 0& 0 &A_1^{\eta^{\ast}}& 0 \\
		E_4^{\eta^{\ast}} & 0 & 0 & -E_4^{\eta^{\ast}} & 0&0& 0 & 0 & 0& 0 &0& A_4^{\eta^{\ast}} \\
		P_2 & A_2 & 0 & 0 & 0&0& 0 & 0 & 0& 0 &0& 0 \\
		C_1 & 0 & A_1 & 0 & 0&0& 0 & 0 & 0& 0 &0& 0 \\
		0 & 0 & 0 & P_3 & A_3&0& 0 & 0 & 0& 0 &0& 0 \\
		0 & 0 & 0 & C_1 & 0&A_1& 0 & 0 & 0& 0 &0& 0 \\
		0 & 0 & 0 & P_4 & 0&0& A_4 & 0 & 0& 0 &0& 0 \\
		\end{array}
		\right)\\
		&=2r\begin{pmatrix}
		E_2& E_{1} & 0 &0 & E_{4} \\
		0 & 0& E_{3} & E_{1}& E_4 \\
		A_{2} & 0 & 0 &0&0\\
		0 & A_{1} & 0 &0&0\\
		0 & 0 & A_{3} &0&0\\
		0 & 0 & 0 &A_{1}&0\\
		0 & 0 & 0 &0&A_{4}
		\end{pmatrix},
		\end{align*}
where $P_i=C_iE_{i}^{\eta^{\ast}} (i=\overline{2,4})$. In this case, the general solution to the system \eqref{eq7-2} can be expressed as
\begin{align*}
&U=\frac{U_1+(U_2)^{\eta^{\ast}}}{2},\ X=\frac{\tilde{X}+(\tilde{X})^{\eta^{\ast}}}{2},\\ &Y=\frac{\tilde{Y}+(\tilde{Y})^{\eta^{\ast}}}{2},\ Z=\frac{\tilde{Z}+(\tilde{Z})^{\eta^{\ast}}}{2},\\
&U_1=A_{1}^{\dagger}C_{1}+L_{A_{1}}S_1,\ U_2=U_1^{\eta^{\ast}},\\
\end{align*}
\begin{align*}
&\tilde{X}=A_{1}^{\dagger}B_{1}+L_{A_{1}}B_{1}^{\eta^{\ast}}(A_{1}^{\eta^{\ast}})^{\dagger}+L_{A_{1}}U_{1}(L_{A_{1}})^{\eta^{\ast}},\\
&\tilde{Y}=A_{2}^{\dagger}B_{2}+L_{A_{2}}B_{2}^{\eta^{\ast}}(A_{2}^{\eta^{\ast}})^{\dagger}+L_{A_{2}}U_{2}(L_{A_{2}})^{\eta^{\ast}},\\ &\tilde{Z}=A_{3}^{\dagger}B_{3}+L_{A_{3}}B_{3}^{\eta^{\ast}}(A_{3}^{\eta^{\ast}})^{\dagger}+L_{A_{3}}U_{3}(L_{A_{3}})^{\eta^{\ast}},
\end{align*}
where
	\begin{align*}
	&S_{1}=A_{11}^{\dagger}(T_1-A_{22}XA_{22}^{\eta^{\ast}}-A_{33}YA_{33}^{\eta^{\ast}}-A_{44}ZA_{44}^{\eta^{\ast}})-A_{11}^{\dagger}W_{11}A_{11}^{\eta^{\ast}}+L_{A_{11}}W_{12},\\
	&U_{1}=A_{12}^{\dagger}T(A_{12}^{\dagger})^{\eta^{\ast}}-A_{12}^{\dagger}A_{13}M_{1}^{\dagger}T(A_{12}^{\dagger})^{\eta^{\ast}}-A_{12}^{\dagger}S_1U_{4}R_{M_{1}^{\eta^{\ast}}}A_{13}^{\dagger}(A_{12}^{\dagger})^{\eta^{\ast}}-A_{12}^{\dagger}S_1A_{13}^{\dagger}T(M_1^{\eta^{\ast}})^{\dagger}A_{13}^{\dagger}(A_{12}^{\dagger})^{\eta^{\ast}}\\
	&+L_{A_{12}}U_{5}+U_{6}R_{A_{12}^{\eta^{\ast}}},\\
	&U_{2}=M_{1}^{\dagger}T(A_{13}^{\dagger})^{\eta^{\ast}}+S_{1}^{\dagger}S_{1}A_{13}^{\dagger}T(M_{1}^{\dagger})^{\eta^{\ast}}+L_{M_{1}}L_{S_{1}}U_{7}+U_{8}R_{A_{13}^{\eta^{\ast}}}+L_{M_{1}}U_{4}R_{M_{1}^{\eta^{\ast}}},\\
	&U_{3}=F_{1}+L_{G_{2}}V_{1}+V_{2}R_{G_{4}^{\eta^{\ast}}}+L_{G_{1}}V_{3}R_{G_{3}^{\eta^{\ast}}},\ or\ U_{3}=F_{2}-L_{G_{4}}W_{1}-W_{2}R_{G_{2}^{\eta^{\ast}}}-L_{G_{3}}W_{3}R_{G_{1}^{\eta^{\ast}}},\\
	&V_{1}=(I_{m},\ 0)\left[C_{11}^{\dagger}(F-C_{22}V_{3}C_{33}^{\eta^{\ast}}-C_{33}W_{3}C_{22}^{\eta^{\ast}})\right]-(I_{m},\ 0)\left[C_{11}^{\dagger}U_{11}C_{11}^{\eta^{\ast}}+L_{C_{11}}U_{12}\right],\\
	&W_{1}=(0,\ I_{m})\left[C_{11}^{\dagger}(F-C_{22}V_{3}C_{33}^{\eta^{\ast}}-C_{33}W_{3}C_{22}^{\eta^{\ast}})\right]-(0,\ I_{m})\left[C_{11}^{\dagger}U_{11}C_{11}^{\eta^{\ast}}+L_{C_{11}}U_{12}\right],\\
	&W_{2}=\left[R_{C_{11}}(F-C_{22}V_{3}C_{33}^{\eta^{\ast}}-C_{33}W_{3}C_{22}^{\eta^{\ast}})(C_{11}^{\eta^{\ast}})^{\dagger}\right]\left(
	\begin{array}{c}
	0 \\
	I_{n} \\
	\end{array}
	\right)+\left[C_{11}C_{11}^{\dagger}U_{11}+U_{21}L_{C_{11}}^{\eta^{\ast}}\right]\left(
	\begin{array}{c}
	0 \\
	I_{n} \\
	\end{array}
	\right),\\
	&V_{2}=R_{C_{11}}(F-C_{22}V_{3}C_{33}^{\eta^{\ast}}-C_{33}W_{3}C_{22}^{\eta^{\ast}})(C_{11}^{\eta^{\ast}})^{\dagger}\left(
	\begin{array}{c}
	0 \\
	I_{n} \\
	\end{array}
	\right)+\left[C_{11}C_{11}^{\dagger}U_{11}+U_{21}L_{C_{11}}^{\eta^{\ast}}\right]\left(
	\begin{array}{c}
	I_{n} \\
	0 \\
	\end{array}
	\right),\\
	&V_{3}=E_{11}^{\dagger}F(E_{22}^{\eta^{\ast}})^{\dagger}-E_{11}^{\dagger}E_{22}M^{\dagger}F(E_{22}^{\eta^{\ast}})^{\dagger}-E_{11}^{\dagger}SE_{22}^{\dagger}FN^{\dagger}E_{11}^{\eta^{\ast}}(E_{22}^{\eta^{\ast}})^{\dagger}-E_{11}^{\dagger}SU_{31}R_{N}E_{11}^{\eta^{\ast}}(E_{22}^{\eta^{\ast}})^{\dagger}\\
	&+L_{E_{11}}U_{32}+U_{33}L_{E_{22}}^{\eta^{\ast}},\\
	&W_{3}=M^{\dagger}F(E_{11}^{\eta^{\ast}})^{\dagger}+S^{\dagger}SE_{22}^{\dagger}FN^{\dagger}+L_{M}L_{S}U_{41}+L_{M}U_{31}R_{N}-U_{42}L_{E_{11}}^{\eta^{\ast}},
	\end{align*}
where $T=T_{1}-A_{33}U_{3}(A_{33})^{\eta^{\ast}}$, $U_{j} (j=4,5,6)$, $U_{i1} (i=1,2,3,4)$, $U_{12}$, $U_{32}$, $U_{33}$, and $U_{42}$ are any matrices with suitable dimensions over $\mathbb{H}$.
\end{theorem}		
\begin{proof}
	Since the solvability of the system \eqref{eq7-2} is equivalent to system
	\begin{equation}\label{eq4.2}
	\begin{aligned}
	&A_{1}U_1=B_{1},\ U_2(A_{1})^{\eta^{\ast}}=B_{1}^{\eta^{\ast}},\ U_2=(U_1)^{\eta^{\ast}},\\
	&A_{2}\tilde{X}=B_{2},\ \tilde{X}(A_{2})^{\eta^{\ast}}=B_{2}^{\eta^{\ast}} ,\ \tilde{X}=\tilde{X}^{\eta^{\ast}},\\
	&A_{3}\tilde{Y}=B_{3},\ \tilde{Y}(A_{3})^{\eta^{\ast}}=B_{3}^{\eta^{\ast}},\ \tilde{Y}=\tilde{Y}^{\eta^{\ast}},\\
	&A_{4}\tilde{Z}=B_{4},\ \tilde{Z}(A_{4})^{\eta^{\ast}}=B_{4}^{\eta^{\ast}},\ \tilde{Z}=\tilde{Z}^{\eta^{\ast}},\\
	E_1U_1&+U_2E_1^{\eta^{\ast}}+E_{2}\tilde{X}E_{2}^{\eta^{\ast}}+E_{3}\tilde{Y}E_{3}^{\eta^{\ast}}+E_{4}\tilde{Z}E_{4}^{\eta^{\ast}}=C_c.
	\end{aligned}
	\end{equation}
	If the system \eqref{eq7-2} has a solution, say, $(U,\ X,\ Y,\ Z)$, then
	$$
	(U_1,\ U_2,\ \tilde{X},\ \tilde{Y},\ \tilde{Z}):=(U,\ U^{\eta^{\ast}},\ X,\ Y,\ Z)
	$$
	is a solution to the system of matrix equations \eqref{eq4.2}. Conversely, if the system \eqref{eq4.2} has a solution, say
	$$
	(U_1,\ U_2,\ \tilde{X},\ \tilde{Y},\ \tilde{Z}),
	$$
	then equations \eqref{eq7-2} clearly has a solution	
\begin{align*}
	&(U,\ X,\ Y,\ Z):\\
	&=(\frac{U_1+(U_2)^{\eta^{\ast}}}{2}\frac{\tilde{X}+(\tilde{X})^{\eta^{\ast}}}{2},\ \frac{\tilde{Y}+(\tilde{Y})^{\eta^{\ast}}}{2},\ \frac{\tilde{Z}+(\tilde{Z})^{\eta^{\ast}}}{2}).
	\end{align*}
\end{proof}		
Next, we study the special case \eqref{eq7-3} of the matrix equations \eqref{eq7-2}.		
	\begin{theorem}\label{the4.1} Let $ A_{i}, C_{i}, E_{i}\ (i=\overline{1,3})$ and $C$ be given with appropriate size. Set
			\begin{align*}
			&E_{1}L_{A_{1}}=A_{11},\ E_{2}L_{A_{2}}=A_{22},\ E_{3}L_{A_{3}}=A_{33},\ R_{A_{11}}A_{22}=M_{1},\ S_{1}=A_{22}L_{M_{1}},\\ &T_{1}=C-E_{1}(A_{1}^{\dagger}B_{1}+L_{A_{1}}B_{1}^{\eta^{\ast}}(A_{1}^{\eta^{\ast}})^{\dagger})E_{1}^{\eta^{\ast}}-E_{2}(A_{2}^{\dagger}B_{2}+L_{A_{2}}B_{2}^{\eta^{\ast}}(A_{2}^{\eta^{\ast}})^{\dagger})E_{2}^{\eta^{\ast}}\\
			&-E_{3}(A_{3}^{\dagger}B_{3}+L_{A_{3}}B_{3}^{\eta^{\ast}}(A_{3}^{\eta^{\ast}})^{\dagger})E_{3}^{\eta^{\ast}},\ G=R_{M_{1}}R_{A_{11}},\\
			& G_{1}=GA_{33} ,\ G_{2}=R_{A_{11}}A_{33},\ G_{3}=R_{A_{22}}A_{33},\ G_{4}=A_{33},\\
			&L_{1}=GT_{1},\ L_{2}=R_{A_{11}}T_{1}(R_{A_{22}})^{\eta^{\ast}},\ L_{3}=R_{A_{22}}T_{1}(R_{A_{11}})^{\eta^{\ast}},\\
			&L_{4}=T_{1}G^{\eta^{\ast}},\ C_{11}=(L_{C_{2}},\ L_{C_{4}}),\ C_{22}=L_{C_{1}},\ C_{33}=L_{C_{3}},\\
			&E_{11}=R_{C_{11}}C_{22},\ E_{22}=R_{C_{11}}C_{33},\ M=R_{E_{11}}E_{22},\\
			& N=(R_{E_{22}}E_{11})^{\eta^{\ast}},\ F=F_{2}-F_{1},\ E=R_{C_{11}}F(R_{C_{11}})^{\eta^{\ast}},\\
			& S=E_{22}L_{M},\ F_{11}=G_{2}L_{G_{1}},\ H_{1}=L_{2}-G_{2}G_{1}^{\dagger}L_{1}(G_{4}^{\eta^{\ast}})^{\dagger}G_{3}^{\eta^{\ast}},\\
			& F_{22}=G_{4}L_{G_{3}},\ H_{2}=L_{4}-G_{4}G_{3}^{\dagger}L_{3}(G_{2}^{\eta^{\ast}})^{\dagger}G_{1}^{\eta^{\ast}},\\
			&F_{1}=G_{1}^{\dagger}L_{1}(G_{4}^{\eta^{\ast}})^{\dagger}+L_{G_{1}}G_{2}^{\dagger}L_{2}(G_{3}^{\eta^{\ast}})^{\dagger},\ F_{2}=G_{3}^{\dagger}L_{3}(G_{2}^{\eta^{\ast}})^{\dagger}+L_{G_{3}}G_{4}^{\dagger}L_{4}(G_{1}^{\eta^{\ast}})^{\dagger}.
			\end{align*}
			Then the following statements are equivalent:
			
			$\mathrm{(1)}$ The system of the matrix equations \eqref{eq7-3} is consistent.
			
			$\mathrm{(2)}$
			\begin{align*}
			& R_{A_{j}}B_{j}=0,\ R_{G_{i}}L_{i}=0 (i=\overline{1,4}, j=\overline{1,3}),\\
			& R_{E_{22}}E(R_{E_{22}})^{\eta^{\ast}}=0.
			\end{align*}
			$\mathrm{(3)}$
			\begin{align*}
			&r(A_{j},\ B_{j})=r(A_{j})(j=1,2,3),\\
			&r\left(
			\begin{array}{cccc}
			C & E_{3} & E_{1} & E_{2} \\
			B_{3}E_{3}^{\eta^{\ast}} & A_{3} & 0 & 0 \\
			B_{1}E_{1}^{\eta^{\ast}} & 0 & A_{1} & 0 \\
			B_{2}E_{2}^{\eta^{\ast}} & 0 & 0 & A_{2} \\
			\end{array}
			\right)=r\left(
			\begin{array}{ccc}
			E_{3} & E_{1} & E_{2} \\
			A_{3} & 0 & 0 \\
			0 & A_{1} & 0 \\
			0 & 0 & A_{2} \\
			\end{array}
			\right),\\
\end{align*}
\begin{align*}
			&r\left(
			\begin{array}{cccc}
			C & E_{3} & E_{1} & E_{2}B_{2}^{\eta^{\ast}} \\
			E_{2}^{\eta^{\ast}} & 0 & 0 & A_{2}^{\eta^{\ast}} \\
			B_{3}E_{3}^{\eta^{\ast}} & A_{3} & 0 & 0 \\
			B_{1}E_{1}^{\eta^{\ast}} & 0 & A_{1} & 0 \\
			\end{array}
			\right)=r\left(
			\begin{array}{cc}
			E_{3} & E_{1} \\
			A_{3} & 0 \\
			0 & A_{1} \\
			\end{array}
			\right)+r\left(
			\begin{array}{c}
			E_{2} \\
			A_{2} \\
			\end{array}
			\right),\\
			&r\left(
			\begin{array}{cccc}
			C & E_{3} & E_{2} & E_{1}B_{1}^{\eta^{\ast}} \\
			E_{1}^{\eta^{\ast}} & 0 & 0 & A_{1}^{\eta^{\ast}} \\
			B_{3}E_{3}^{\eta^{\ast}} & A_{3} & 0 & 0 \\
			B_{2}E_{2}^{\eta^{\ast}} & 0 & A_{2} & 0 \\
			\end{array}
			\right)=r\left(
			\begin{array}{cc}
			E_{3} & E_{2} \\
			A_{3} & 0 \\
			0 & A_{2} \\
			\end{array}
			\right)+r\left(
			\begin{array}{c}
			E_{1} \\
			A_{1} \\
			\end{array}
			\right),\\
			&r\begin{pmatrix}
			C & E_{3} & E_{1}B_{1}^{\eta^{\ast}} & E_{2}B_{2}^{\eta^{\ast}} \\
			E_{1}^{\eta^{\ast}} & 0 & A_{1}^{\eta^{\ast}} & 0 \\
			E_{2}^{\eta^{\ast}} & 0 & 0 & A_{2}^{\eta^{\ast}} \\
			B_{3}E_{3}^{\eta^{\ast}} & A_{3} & 0 & 0 \\
			\end{pmatrix}=r
			\begin{pmatrix}
			E_{1} & E_{2} \\
			A_{1} & 0 \\
			0 & A_{2} \\
			\end{pmatrix}+r\begin{pmatrix}
			E_{3} \\
			A_{3} \\
			\end{pmatrix},\\
			&r\left(
			\begin{array}{cccccccc}
			C & 0 & E_{1} & 0 & E_{3} & E_{2}B_{2}^{\eta^{\ast}} & 0 & E_{3}B_{3}^{\eta^{\ast}} \\
			0 & -C & 0 & E_{2} & E_{3} & 0 & -E_{1}B_{1}^{\eta^{\ast}} & 0 \\
			E_{2}^{\eta^{\ast}} & 0 & 0 & 0 & 0 & A_{2}^{\eta^{\ast}} & 0 & 0 \\
			0 & E_{1}^{\eta^{\ast}} & 0 & 0 & 0 & 0 & A_{1}^{\eta^{\ast}} & 0 \\
			E_{3}^{\eta^{\ast}} & E_{3}^{\eta^{\ast}} & 0 & 0 & 0 & 0 & 0 & A_{3}^{\eta^{\ast}} \\
			B_{1}E_{1}^{\eta^{\ast}} & 0 & A_{1} & 0 & 0 & 0 & 0 & 0 \\
			0 & -B_{2}E_{2}^{\eta^{\ast}} & 0 & A_{2} & 0 & 0 & 0 & 0 \\
			0 & -B_{3}E_{3}^{\eta^{\ast}} & 0 & 0 & A_{3} & 0 & 0 & 0 \\
			\end{array}
			\right)\\
			&=2r\left(
			\begin{array}{ccc}
			E_{1} & 0 & E_{3} \\
			0 & E_{2} & E_{3}  \\
			A_{1} & 0 & 0  \\
			0 & A_{2} & 0  \\
			0 & 0 & A_{3}  \\
			\end{array}
			\right).
			\end{align*}	
In this case, the general solution of matrix equation \eqref{eq7-3} can be expressed as
			\begin{align*}
			&X=\frac{\tilde{X}+(\tilde{X})^{\eta^{\ast}}}{2},\ Y=\frac{\tilde{Y}+(\tilde{Y})^{\eta^{\ast}}}{2},\ Z=\frac{\tilde{Z}+(\tilde{Z})^{\eta^{\ast}}}{2},\\
			&\tilde{X}=A_{1}^{\dagger}B_{1}+L_{A_{1}}B_{1}^{\eta^{\ast}}(A_{1}^{\eta^{\ast}})^{\dagger}+L_{A_{1}}U_{1}(L_{A_{1}})^{\eta^{\ast}},\\
			&\tilde{Y}=A_{2}^{\dagger}B_{2}+L_{A_{2}}B_{2}^{\eta^{\ast}}(A_{2}^{\eta^{\ast}})^{\dagger}+L_{A_{2}}U_{2}(L_{A_{2}})^{\eta^{\ast}},\\
			&\tilde{Z}=A_{3}^{\dagger}B_{3}+L_{A_{3}}B_{3}^{\eta^{\ast}}(A_{3}^{\eta^{\ast}})^{\dagger}+L_{A_{3}}U_{3}(L_{A_{3}})^{\eta^{\ast}},
			\end{align*}
			where	
			\begin{align*}
			&U_{1}=A_{11}^{\dagger}T(A_{11}^{\dagger})^{\eta^{\ast}}-A_{11}^{\dagger}A_{22}M_{1}^{\dagger}T(A_{11}^{\dagger})^{\eta^{\ast}}-A_{11}^{\dagger}U_{4}A_{22}^{\dagger}T(M_{1}^{\dagger})^{\eta^{\ast}}(A_{22}^{\dagger})^{\eta^{\ast}}+L_{A_{11}}U_{5}+U_{6}R_{A_{11}^{\eta^{\ast}}},\\
			&U_{2}=M_{1}^{\dagger}T(A_{22}^{\dagger})^{\eta^{\ast}}+S_{1}^{\dagger}S_{1}A_{22}^{\dagger}T(M_{1}^{\dagger})^{\eta^{\ast}}+L_{M_{1}}L_{S_{1}}U_{7}+U_{8}R_{A_{22}^{\eta^{\ast}}}+L_{M_{1}}U_{4}R_{M_{1}^{\eta^{\ast}}},\\
			&U_{3}=F_{1}^{0}+L_{G_{2}}V_{1}+V_{2}R_{G_{4}^{\eta^{\ast}}}+L_{G_{1}}V_{3}R_{G_{3}^{\eta^{\ast}}},\ or \ U_{3}=F_{2}^{0}-L_{G_{4}}W_{1}-W_{2}R_{G_{2}^{\eta^{\ast}}}-L_{G_{3}}W_{3}R_{G_{1}^{\eta^{\ast}}},\\
		\end{align*}
			\begin{align*}	
&V_{1}=(I_{m},\ 0)\left[C_{11}^{\dagger}(F-C_{22}V_{3}C_{33}^{\eta^{\ast}}-C_{33}W_{3}C_{22}^{\eta^{\ast}})\right]-(I_{m},\ 0)\left[C_{11}^{\dagger}U_{11}C_{11}^{\eta^{\ast}}+L_{C_{11}}U_{12}\right],\\
			&W_{1}=(0,\ I_{m})\left[C_{11}^{\dagger}(F-C_{22}V_{3}C_{33}^{\eta^{\ast}}-C_{33}W_{3}C_{22}^{\eta^{\ast}})\right]-(0,\ I_{m})\left[C_{11}^{\dagger}U_{11}C_{11}^{\eta^{\ast}}+L_{C_{11}}U_{12}\right],\\
			&W_{2}=\left[R_{C_{11}}(F-C_{22}V_{3}C_{33}^{\eta^{\ast}}-C_{33}W_{3}C_{22}^{\eta^{\ast}})(C_{11}^{\eta^{\ast}})^{\dagger}\right]\left(
			\begin{array}{c}
			0 \\
			I_{n} \\
			\end{array}
			\right)+\left[C_{11}C_{11}^{\dagger}U_{11}+U_{21}L_{C_{11}}^{\eta^{\ast}}\right]\left(
			\begin{array}{c}
			0 \\
			I_{n} \\
			\end{array}
			\right),\\
			&V_{2}=R_{C_{11}}(F-C_{22}V_{3}C_{33}^{\eta^{\ast}}-C_{33}W_{3}C_{22}^{\eta^{\ast}})(C_{11}^{\eta^{\ast}})^{\dagger}\left(
			\begin{array}{c}
			0 \\
			I_{n} \\
			\end{array}
			\right)+\left[C_{11}C_{11}^{\dagger}U_{11}+U_{21}L_{C_{11}}^{\eta^{\ast}}\right]\left(
			\begin{array}{c}
			I_{n} \\
			0 \\
			\end{array}
			\right),\\
			&V_{3}=E_{11}^{\dagger}F(E_{22}^{\eta^{\ast}})^{\dagger}-E_{11}^{\dagger}E_{22}M^{\dagger}F(E_{22}^{\eta^{\ast}})^{\dagger}-E_{11}^{\dagger}SE_{22}^{\dagger}FN^{\dagger}E_{11}^{\eta^{\ast}}(E_{22}^{\eta^{\ast}})^{\dagger}-E_{11}^{\dagger}SU_{31}R_{N}E_{11}^{\eta^{\ast}}(E_{22}^{\eta^{\ast}})^{\dagger}\\
			&+L_{E_{11}}U_{32}+U_{33}L_{E_{22}}^{\eta^{\ast}},\\
			&W_{3}=M^{\dagger}F(E_{11}^{\eta^{\ast}})^{\dagger}+S^{\dagger}SE_{22}^{\dagger}FN^{\dagger}+L_{M}L_{S}U_{41}+L_{M}U_{31}R_{N}-U_{42}L_{E_{11}}^{\eta^{\ast}},
			\end{align*}
			where $T=T_{1}-A_{33}U_{3}(A_{33})^{\eta^{\ast}}$, $U_{j} (j=4,5,6)$, $U_{i1} (i=1,2,3,4)$, $U_{12}$, $U_{32}$, $U_{33}$ and $U_{42}$ are any matrices with appropriate dimensions.
\end{theorem}
\begin{proof}
	It follows from Theorem 4.1 that this theorem holds when $A_1, C_1$, and $E_1$ vanish in Theorem 4.1.	
	\end{proof}	
		In Theorem 4.2, let $A_{i},\ C_i,\ B_i,\ D_i (i=\overline{1,3})$, $E_3$ and $F_3$ be vanish. Then we can get the $\eta$-Hermitian solution of the matrix equation \eqref{eqn1}.
		
		\begin{corollary}
			Let $B_{1}, C_{1}$ and $D_{1}=D_{1}^{\eta^{\ast}}$ be given. Set $M=R_{B_{1}} C_{1}, S=C_{1} L_{M}$. Then the following statements are equivalent:
			
			$\mathrm{(1)}$ Matrix equation \eqref{eqn1} has a pair of $\eta$-Hermitian solutions $Y$ and $Z$.
			
			$\mathrm{(2)}$
			$$
			R_{M} R_{B_{1}} D_{1}=0, \quad R_{B_{1}} D_{1}\left(R_{C_{1}}\right)^{\eta^{\ast}}=0 .
			$$
			
			$\mathrm{(3)}$	
	\begin{align*}
			&r\begin{pmatrix}
			B_{1} & D_{1} \\
			0 & C_{1}^{\eta^{\ast}}
			\end{pmatrix}=r\left(B_{1}\right)+r\left(C_{1}\right),\\
			& r\begin{pmatrix}
			B_{1} & C_{1} & D_{1}
			\end{pmatrix}=r\begin{pmatrix}
			B_{1} & C_{1}
			\end{pmatrix} .
			\end{align*}
			In this case, the $\eta$-Hermitian solution to matrix equation \eqref{eqn1} can be expressed as
			\begin{align*}
			&Y=B_{1}^{\dagger} D_{1}\left(B_{1}^{\dagger}\right)^{\eta^{\ast}}-\frac{1}{2} B_{1}^{\dagger} C_{1} M^{\dagger} D_{1}\left[I+\left(C_{1}^{\dagger}\right)^{\eta^{\ast}} S^{\eta^{\ast}}\right]\left(B_{1}^{\dagger}\right)^{\eta^{\ast}} \\
			&-\frac{1}{2} B_{1}^{\dagger}\left(I+S C_{1}^{\dagger}\right) D_{1}\left(M^{\dagger}\right)^{\eta^{\ast}} C_{1}^{\eta^{\ast}}\left(B_{1}^{\dagger}\right)^{\eta^{\ast}}-B_{1}^{\dagger} S W_{2} S^{\eta^{\ast}}\left(B_{1}^{\dagger}\right)^{\eta^{\ast}}+L_{B_{1}} U+U^{\eta^{\ast}}\left(L_{B_{1}}\right)^{\eta^{\ast}}, \\
			&Z= \frac{1}{2} M^{\dagger} D_{1}\left(C_{1}^{\dagger}\right)^{\eta^{\ast}}\left[I+\left(S^{\dagger} S\right)^{\eta^{\ast}}\right]+\frac{1}{2}\left(I+S^{\dagger} S\right) C_{1}^{\dagger} D_{1}\left(M^{\dagger}\right)^{\eta^{\ast}} \\
			&+L_{M} W_{2}\left(L_{M}\right)^{\eta^{\ast}}+V L_{C_{1}}^{\eta^{\ast}}+L_{C_{1}} V^{\eta^{\ast}}+L_{M} L_{S} W_{1}+W_{1}^{\eta^{\ast}}\left(L_{S}\right)^{\eta^{\ast}}\left(L_{M}\right)^{\eta^{\ast}},
			\end{align*}
			where $W_{1}, U, V$ and $W_{2}=W_{2}^{\eta^{\ast}}$ are arbitrary matrices over $\mathbb{H}$ with appropriate sizes.
		\end{corollary}
\noindent{\bf Remark 4.3.}\quad The above corollary has the main findings of \cite{H.W. 2013}.
		
		\begin{corollary}
			Let $A_{1}, C_{1}, A_{2}, A_{3}, B_1, D_1, D_3$ and $D_{3}=D_{3}^{\eta^{\ast}}$ be coefficient matrices in \eqref{eqn2}. Define some new matrices as follows:
			\begin{align*}
			& B_{4}=A_{2} L_{A_{1}}, \ C_{4}=A_{3}\left(R_{B_{1}}\right)^{\eta^{\ast}}, \\
			&D_{4} =D_{3}-A_{2}\left[A_{1}^{\dagger} C_{1}+\left(A_{1}^{\dagger} C_{1}\right)^{\eta *}-A_{1}^{\dagger} A_{1} C_{1}^{\eta^{\ast}}\left(A_{1}^{\dagger}\right)^{\eta *}\right] A_{2}^{\eta^{\ast}} \\
			&-A_{3}\left[D_{1} B_{1}^{\dagger}+\left(D_{1} B_{1}^{\dagger}\right)^{\eta^{\ast}}-\left(B_{1}^{\dagger}\right)^{\eta^{\ast}} B_{1}^{\eta^{\ast}} D_{1} B_{1}^{\dagger}\right] A_{3}^{\eta^{\ast}},\\
			&	M=R_{B_{4}}C_{4}, \quad S=C_{4} L_{M} .
			\end{align*}	
			Then the following statements are equivalent:
			
			$\mathrm{(1)}$ The system \eqref{eqn2} has a solution $(X, Y, Z)$, where $Y$ and $Z$ are $\eta$-Hermitian.
			
			$\mathrm{(2)}$ The coefficient matrices in equations \eqref{eqn2} satisfy
			\begin{align*}
			&A_{1} C_{1}^{\eta^{\ast}}=C_{1} A_{1}^{\eta^{\ast}}, \quad B_{1}^{\eta^{\ast}} D_{1}=D_{1}^{\eta^{\ast}} B_{1}, \\
			&R_{A_{1}} C_{1}=0, \quad D_{21} L_{B_{1}}=0, \quad R_{M} R_{B_4} D_4=0,\\
			& \quad R_{B_4} D_4\left(R_{C_4}\right)^{\eta^{\ast}}=0 .
			\end{align*}
			$\mathrm{(3)}$ The coefficient matrices in equations \eqref{eqn2} and their ranks satisfy
			\begin{align*}
			&A_{1} C_{1}^{\eta^{\ast}}=C_{1} A_{1}^{\eta^{\ast}}, \quad B_{1}^{\eta^{\ast}} D_{1}=D_{1}^{\eta^{\ast}} B_{1}, \\
			&r\left(\begin{array}{ll}
			A_{1} & C_{1}
			\end{array}\right)=r\left(A_{1}\right), \quad r\left(\begin{array}{c}
			D_{1} \\
			B_{1}
			\end{array}\right)=r\left(B_{1}\right),\\
			&r\left(\begin{array}{ccc}
			D_{3} & A_{3} & A_{2}  \\
			D_{1}^{\eta^{\ast}} A_{3}^{\eta^{\ast}} & B_{1}^{\eta^{\ast}} & 0\\
			C_{1} A_{2}^{\eta^{*}} & 0 & A_{1}  \\
			C_{1} & 0 & 0
			\end{array}\right)=r\left(\begin{array}{cc}
			A_{3} & A_{2}  \\
			B_{1}^{\eta *} & 0  \\
			0 & A_{1} \\
			\end{array}\right)+r\left(\begin{array}{l}
			A_{1}
			\end{array}\right),\\
			&r\left(\begin{array}{ccc}
			D_{3} & A_{2} &  A_{3} D_{1} \\
			A_{3}^{\eta^{\ast}}  & 0 & B_{1} \\
			C_{1} A_{2}^{\eta *} & A_{1} & 0  \\
			\end{array}\right)=r\left(\begin{array}{c}
			A_{2}  \\
			A_{1}
			\end{array}\right)+r\left(\begin{array}{cc}
			A_{3}^{\eta^{\ast}}  & B_{1}
			\end{array}\right).
			\end{align*}
			In this case, the general solution to the system of matrix equations \eqref{eqn2} can be expressed as
\begin{align*}
			&Y=Y^{\eta^{\ast}}=A_{1}^{\dagger} C_{1}+\left(A_{1}^{\dagger} C_{1}\right)^{\eta^{\ast}}-A_{1}^{\dagger} A_{1} C_{1}^{\eta^{\ast}}\left(A_{1}^{\dagger}\right)^{\eta^{\ast}}\\
			&+L_{A_{1}} V\left(L_{A_{1}}\right)^{\eta^{\ast}}, \\
			&Z=Z^{\eta^{\ast}}=D_{1} B_{1}^{\dagger}+\left(D_{1} B_{1}^{\dagger}\right)^{\eta^{\ast}}-\left(B_{1}^{\dagger}\right)^{\eta^{\ast}} B_{1}^{\eta^{\ast}} D_{1} B_{1}^{\dagger}\\
			&+\left(R_{B_{1}}\right)^{\eta^{\ast}} W R_{B_{1}},\\
			&V=V^{\eta^{\ast}}=B_{4}^{\dagger} D_4\left(B_{4}^{\dagger}\right)^{\eta^{\ast}}\\
			&-\frac{1}{2} B_{4}^{\dagger} C_4 M^{\dagger} D_4\left[I+\left(C_{4}^{\dagger}\right)^{\eta^{\ast}} S^{\eta *}\right]\left(B_{4}^{\dagger}\right)^{\eta^{\ast}}\\
				\end{align*}
			\begin{align*}
			&-\frac{1}{2} B_{4}^{\dagger}\left(I+S C_{4}^{\dagger}\right) D_4\left(M^{\dagger}\right)^{\eta^{\ast}} C_{4}^{\eta^{\ast}}\left(B_{4}^{\dagger}\right)^{\eta^{\ast}}\\
			&-B_{4}^{\dagger} S U_{6} S^{\eta^{\ast}}\left(B_{4}^{\dagger}\right)^{\eta^{\ast}}+L_{B_{4}} U_{4}+U_{4}^{\eta^{\ast}}\left(L_{B_{4}}\right)^{\eta^{\ast}},\\
			&W=W^{\eta^{\ast}}=\frac{1}{2} M^{\dagger} D_4\left(B^{\dagger}\right)^{\eta^{\ast}}\left[I+\left(S^{\dagger} S\right)^{\eta}\right]\\
			&+\frac{1}{2}\left(I+S^{\dagger} S\right) C_{4}^{\dagger} D_4\left(M^{\dagger}\right)^{\eta *}+L_{M} U_{6}\left(L_{M}\right)^{\eta}+U_{5} L_{C_{4}}^{\eta}\\
			&+L_{C_{4}} U_{5}^{\eta^{*}}+L_{M} L_{5} U_{3}+U_{3}^{\eta *}\left(L_{S}\right)^{\eta}\left(L_{M}\right)^{\eta^{\ast}},
			\end{align*}
			where $U_{3}, \ U_4\ , U_{5}$ and $U_{6}=U_{6}^{\eta^{\ast}}$ are arbitrary matrices over $\mathbb{H}$ with appropriate sizes.
\end{corollary}
\section{Algorithm with a numerical example}
	In this section, we present an algorithm and an example to illustrate Theorem 3.1.
	
	$\textbf{Algorithm 5.1}$
	
	(1) Feed the values of $A_{i}, B_{i}\ C_{i},\ D_{i},\ E_{i},\ F_{i} (i=\overline{1,4})$ and $C_c$ with conformable shapes over $\mathbb{H}$.
	
	(2) Compute the symbols in \eqref{3.1} to \eqref{3.4}.
	
	(3) Check (2) in Theorem 3.1 or  \eqref{3.6} to \eqref{3.15}. If no, it returns ``inconsisten".
	
	(4) Else, compute $U\ V,\ X,\ Y,\ Z$.
	
	\textbf{Example 5.1} Let
	\begin{align*}
	&A_{1}=\left(
	\begin{array}{ccc}
	1 & 0 & \mathbf{i}
	\end{array}
	\right),\ B_{1}=\left(
	\begin{array}{c}
	0  \\
	1 \\
	\mathbf{j}
	\end{array}
	\right),\ A_{2}=\left(
	\begin{array}{cc}
	0 & \mathbf{i} \\
	\mathbf{j} & \mathbf{k} \\
	\end{array}
	\right),\ B_{2}=\left(
	\begin{array}{c}
	\mathbf{i} \\
	1
	\end{array}
	\right),\\
	& A_{3}=\left(
	\begin{array}{cc}
	0 & 1 \\
	\mathbf{i} & \mathbf{j} \\
	\end{array}
	\right),\ B_{3}=\left(
	\begin{array}{c}
	\mathbf{j} \\
	1
	\end{array}\right),\ A_{4}=\left(
	\begin{array}{cc}
	\mathbf{i} & 1 \\
	0 & \mathbf{k} \\
	\end{array}
	\right),\ B_{4}=\left(
	\begin{array}{c}
	1 \\
	\mathbf{k}
	\end{array}
	\right),\\
	&C_{1}=1+2\mathbf{i},\ D_{1}=\left(
	\begin{array}{c}
	\mathbf{j} \\
	-\mathbf{i}+\mathbf{j}
	\end{array}
	\right),\ C_{2}=\left(
	\begin{array}{cc}
	-1 & 3\mathbf{i} \\
	3\mathbf{j} & 3\mathbf{k} \\
	\end{array}
	\right),\ D_{2}=\left(
	\begin{array}{c}
	2\mathbf{i} \\
	2
	\end{array}
	\right),\\
	&C_{3}=\left(
	\begin{array}{cc}
	0 & 2 \\
	\mathbf{i} & -1+2\mathbf{j} \\
	\end{array}
	\right),\ C_{4}=\left(
	\begin{array}{cc}
	-1 & 1+\mathbf{k} \\
	0& \mathbf{k} \\
	\end{array}
	\right),\ D_{3}=\left(
	\begin{array}{c}
	\mathbf{i}+\mathbf{j} \\
	2
	\end{array}
	\right),\\\
		&D_{4}=\left(
		\begin{array}{c}
		2\mathbf{i} \\
		\mathbf{k}
		\end{array}
		\right), \ C_{3}=\left(
		\begin{array}{cc}
		-1 & 1+\mathbf{k} \\
		0& \mathbf{k} \\
		\end{array}
		\right),\  D_{1}=\left(
		\begin{array}{c}
		2\mathbf{i} \\
		2
		\end{array}
		\right),\\
		& D_{2}=\left(
		\begin{array}{c}
		\mathbf{i}+\mathbf{j} \\
		2
		\end{array}
		\right),\ D_{3}=\left(
		\begin{array}{c}
		2\mathbf{i} \\
		\mathbf{k}
		\end{array}
		\right), \ E_{1}=\left(
		\begin{array}{cc}
		2 & \mathbf{i} \\
		0 & \mathbf{k} \\
		\end{array}
		\right),\\
		& E_{2}=\left(
		\begin{array}{cc}
		\mathbf{i} & \mathbf{j} \\
		0 & \mathbf{k} \\
		\end{array}
		\right),\  E_{3}=\left(
		\begin{array}{cc}
		0 & 1\\
		\mathbf{j} & \mathbf{k}
		\end{array}
		\right),\ F_{1}=\left(
		\begin{array}{c}
		\mathbf{i} \\
		\mathbf{j} \\
		\end{array}
		\right),\\
		& F_{2}=\left(
		\begin{array}{c}
		\mathbf{j} \\
		\mathbf{i} \\
		\end{array}
		\right),\ F_{3}=\left(
		\begin{array}{c}
		\mathbf{k} \\
		\mathbf{i} \\
		\end{array}
		\right),\ C=\left(
		\begin{array}{c}
		3\mathbf{i}+2\mathbf{k} \\
		1-4\mathbf{i}+3\mathbf{j}-\mathbf{k} \\
		\end{array}
		\right).
		\end{align*}
Computation  directly  yields
		\begin{align*}
		&A_{1}D_{1}=C_{1}B_{1}=\left(
		\begin{array}{c}
		2\mathbf{i} \\
		0  \\
		\end{array}
		\right),\\ &A_{2}D_{2}=C_{2}B_{2}=\left(
		\begin{array}{c}
		2 \\
		-1+2\mathbf{j}+\mathbf{k} \\
		\end{array}
		\right),\\
		&A_{3}D_{3}=C_{3}B_{3}=\left(
		\begin{array}{c}
		-2+\mathbf{k} \\
		-\mathbf{i} \\
		\end{array}
		\right),\\
		&r(C_{i},\ A_{i})=r(A_{i})=2,\ r\left(
		\begin{array}{c}
		D_{i} \\
		B_{i} \\
		\end{array}
		\right)=r(B_{i})=1 (i=\overline{1,3}),\\
		&\eqref{3.7}=11,\ \eqref{3.8}=8,\ \eqref{3.9}=10,\ \eqref{3.10}=9,\\
		&\eqref{3.11}=10,\ \eqref{3.12}=9,\ \eqref{3.13}=9,\ \eqref{3.14}=8,\ \eqref{3.15}=19.
		\end{align*}
	All the rank equalities in \eqref{3.6} to \eqref{3.15} hold. Hence, according to Theorem 3.1, the system of matrix equations \eqref{eq1} has a solution, and the general solution to matrix equations \eqref{eq1} can be expressed as
	\begin{align*}
	&U=\begin{pmatrix}
	0.5000+1.0000\mathbf{i}\\
	0\\
	1-0.5000\mathbf{i}\end{pmatrix}\\
	&+\begin{pmatrix}
	0.5000&0&-0.5000\mathbf{i}\\
	0&1.000&0\\
	0.50000\mathbf{i}&0&0.5000
	\end{pmatrix}S_1,\\
	&V=\begin{pmatrix}
	0&0.5000\mathbf{j}&0.5000\\
	0&-0.5000\mathbf{i}+0.5000\mathbf{j}&0.5000+0.5000\mathbf{k}
	\end{pmatrix},\\
	& X=\left(
	\begin{array}{cc}
	2.000 & 0 \\
	1.000\mathbf{i} & 3.000 \\
	\end{array}
	\right),\\
	& Y=\left(
	\begin{array}{cc}
	1.000 & 1.000\mathbf{i} \\
	0 & 2.000 \\
	\end{array}
	\right),\ Z=\left(
	\begin{array}{cc}
	1.000\mathbf{i} & 1.000\mathbf{j} \\
	0 & 1.000 \\
	\end{array}
	\right),
	\end{align*}
	where
	\begin{align*}
	&S_1=\begin{pmatrix}
	-0.5000+2.000\mathbf{i}-1.0000\mathbf{k}\\
	1.0000-3.0000\mathbf{i}+1.0000\mathbf{j}+1.0000\mathbf{k}\\
	-2.0000-0.5000\mathbf{i}+1.0000\mathbf{j}
	\end{pmatrix}
	\end{align*}
	\begin{align*}
	&-\begin{pmatrix}
	1.0000&-1.0000\mathbf{i}\\
	-1.0000&1.0000\mathbf{i}+1.0000\\
	1.0000\mathbf{i}&1.0000
	\end{pmatrix}W_{11}\begin{pmatrix}
	2.0000\\
	-1.0000\mathbf{k}\\
	1.0000\mathbf{i}
	\end{pmatrix},
	\end{align*}
	$W_{11}$ is a any matrix equation with suitable size over $\mathbb{H}$.
		
		Finally, we give the following conclusion that summarizes the work of this paper.
\section{Conclusions }
		 We have established the solvability conditions and a formula for the general solution to the Sylvester-type quaternion matrix equations \eqref{eq1}. As an application of equations \eqref{eq1}, we also have established some necessary and sufficient conditions for the system of quaternion matrix equations \eqref{eq7-2} to provide a solution and derived an exact expression of its general solution involving $\eta$-Hermicity. As a special case of equations \eqref{eq1}, we have presented the necessary and sufficient conditions for the system of two-sided Sylvester-type quaternion matrix equations \eqref{eq7-1} to be consistent and derived a formula for its general solution (when it is solvable). As a special case of equations \eqref{eq7-2}, we have investigated the necessary and sufficient conditions for the system of matrix equations \eqref{eq7-3} to have a solution and provided a general solution, which is an $\eta$-Hermitian.
		
		It is noteworthy that the main results of \eqref{eq1} are available over $\mathbb{R}$ and $\mathbb{C}$ and for any division ring. Furthermore, motivated by \cite{Li 2022}, we can investigate equations \eqref{eq1} in tensor form.

\end{document}